\documentclass[10pt]{article}
\usepackage{lmodern}
\usepackage{amsmath}
\usepackage[T1]{fontenc}
\usepackage[utf8]{inputenc}
\usepackage{authblk}
\usepackage{amsfonts}
\usepackage{graphicx}
\usepackage[colorlinks=true,linkcolor=magenta,citecolor=blue]{hyperref}

\usepackage{tikz}
\usetikzlibrary{decorations.pathreplacing,calc}
\usepackage{rotating}
\tikzset{liltext/.style={font=\tiny}}
\usepackage{array}
\usepackage{float}

\usepackage{amssymb}
\usepackage[english]{babel}
\usepackage{amsthm}
\usepackage{graphicx}
\usepackage{mathrsfs}
\usepackage{makecell}
\usepackage{microtype}
\usepackage{enumerate}
\usepackage{mathscinet}
\usepackage{multirow}
\usepackage{booktabs}%table
\usepackage{enumerate}
\usepackage[cal=boondoxo,bb=ams]{mathalfa}
\hypersetup{hidelinks}
\usepackage{titlesec}
\newtheorem{theorem}{Theorem}[section]
\newtheorem{prop}{Proposition}[section]
\newtheorem{lemma}{Lemma}[section]

\newtheorem{remark}{Remark}[section]

\newcommand{\ml}{\mathcal}
\newcommand{\mb}{\mathbb}

\newcommand{\R}{\mathbb{R}}

\newcommand{\e}{\varepsilon}
\newcommand{\ity}{\infty}

\newcommand{\f}{\displaystyle\frac}

\title{The Cauchy problem for the nonlinear viscous Boussinesq equation in the $L^q$ framework} %without Becker's assumption}
\author[1]{Wenhui Chen\thanks{Wenhui Chen (wenhui.chen.math@gmail.com)}}
\author[2,3]{Tuan Anh Dao\thanks{Tuan Anh Dao (anh.daotuan@hust.edu.vn)}}
\affil[1]{School of Mathematical Sciences, Shanghai Jiao Tong University, 200240 Shanghai, China}
\affil[2]{School of Applied Mathematics and Informatics, Hanoi University of Science and Technology, No. 1 Dai Co Viet road, Hanoi, Vietnam}
\affil[3]{Institute of Mathematics, Vietnam Academy of Science and Technology, No. 18 Hoang Quoc Viet road, Hanoi, Vietnam}

\date{}
\begin{document}

\maketitle
\begin{abstract}
	\medskip
 In this paper, we study the viscous Boussinesq equation in the whole space $\mathbb{R}^n$, which describes the propagation of small amplitude and long waves on the surface of water with viscous effects. Concerning the linearized Cauchy problem, some qualitative properties of solutions including $L^m-L^q$ estimates with $1\leqslant m\leqslant q\leqslant \infty$, inviscid limits and asymptotic profiles of solution with respect to the small viscosity are investigated by means of the Fourier analysis and the WKB method. For another, by applying some fractional order interpolations in the harmonic analysis, we derive the $L^q$ well-posedness and estimates for small data solutions to the nonlinear viscous Boussinesq equation under some conditions for the parameter of nonlinearity.
\\ 
	
	
	\noindent\textbf{AMS Classification (2020)}  35G25,  35B40, 35B25, 35A01
\end{abstract}
\fontsize{12}{15}
\selectfont

%\tableofcontents

\section{Introduction} 
\subsection{Background of the inviscid Boussinesq equation}
It is well-known that mathematical studies in the propagation of water waves over a fluid have been widely applied in fluid mechanics, hydrogeology as well as coastal engineering. There are several mathematical models to describe the propagation of water waves. Particularly in the 1870's, originating from the Euler equations, a series of works \cite{Boussinesq-01,Boussinesq-02,Boussinesq-03} from Joseph Boussinesq derived some fundamental models for the propagation of long waves with small amplitude on the surface of shallow water. Among them, the so-called \emph{Boussinesq model} in the inviscid medium, i.e.
\begin{align*}
	u_{tt}-u_{xx}+u_{xxxx}=(u^2)_{xx},
\end{align*}
was introduced  in \cite{Boussinesq-03}, which is innovative to provide some scientific explanations to the phenomenon of permanent waves observed by the report \cite{Russell}. In the above, the scalar unknown function $u=u(t,x)$ may stand for an elevation of the free surface of inviscid fluids.

Later, a general Boussinesq model carrying a nonlinear term $(f(u))_{xx}$ with a smooth, pure power function $f=f(u)$ arises in the modeling of nonlinear strings \cite{Bona-Sachs=1988}. In the multi-dimensional case, the generalization of inviscid Boussinesq equation is introduced by
\begin{align}\label{Generalized_Boussinesq_Eq}
	u_{tt}-\Delta u+\Delta^2u=\Delta f(u;p),
\end{align}
where the smooth function  $f(u;p)$ behaves as a power nonlinearity $f(u;p)=\mathcal{O}(|u|^p)$ as $|u|\downarrow 0$ with $p>1$. Indeed, the inviscid Boussinesq equation \eqref{Generalized_Boussinesq_Eq} models nonlinear waves in the media with dispersion because of its dispersive linear semigroup
\begin{align*}
\exp\left(\pm it|D|\sqrt{1+|D|^2}\,\right).	
\end{align*}

 Until now, some qualitative properties of solutions to \eqref{Generalized_Boussinesq_Eq}, including well-posedness, regularity analysis, blow-up and asymptotic behaviors, have been deeply established. We refer the interested readers to \cite{Levine=1974,Deift-Tomei-Trub=1982,Bona-Sachs=1988,Manoranjan-Ortega-Sanz=1988,Sachs=1990,Tsutsumi-Matahashi=1991,Straughan=1992,Linares=1993,Liu=1995,Linares-Scialom=1995,Liu=1997,Xue=2006,Cho-Ozawa=2007,Yang-Guo=2008,Liu-Xu=2008,Lin-Wu-Loxton=2009,Farah=2009,Farah=2009-02,Farah-Linares=2010,Baro-Figu-Himo=2019} and references therein.

\subsection{Background of the viscous Boussinesq equation}
Although the viscous effect may be neglected for most of oceanic situations, the viscosity that measures its internal resistance to motions, cannot be excluded for surface waves in relatively shallow channels. For this reason, concerning the long wave propagation in the viscous medium, it is nature to consider the following viscous Boussinesq equation:
\begin{align}\label{Generalized_Viscous_Boussinesq_Eq}
	u_{tt}-\Delta u+\Delta^2u- 2\epsilon \Delta u_t=\Delta f(u;p),
\end{align}
where the constant $\epsilon\in(0,1)$ stands for the viscous coefficient.  In mathematic analysis for fluid mechanics, the rigorous justification for inviscid limits, i.e. the convergence of solutions as the viscosity tending to zero, is one of the most challenging and crucial questions, e.g. Navier-Stokes equations versus Euler equations \cite{Masmoudi=2007}. Formally taking $\epsilon=0$ in the viscous model \eqref{Generalized_Viscous_Boussinesq_Eq}, it immediately turns out to be the inviscid model \eqref{Generalized_Boussinesq_Eq}, whose global (in time) inviscid limit is always an open problem. The viscous Boussinesq equation \eqref{Generalized_Viscous_Boussinesq_Eq} depicts nonlinear waves in the media with dissipation as well as dispersion due to its dissipative-dispersive linear semigroup
\begin{align*}
	\exp\left(-\epsilon t|D|^2\pm it|D|\sqrt{1+(1-\epsilon^2)|D|^2}\,\right)
\end{align*}
We now may observe that a crucial difference between \eqref{Generalized_Boussinesq_Eq} and \eqref{Generalized_Viscous_Boussinesq_Eq} is the dissipation component $\exp(-\epsilon t|D|^2)$ caused by the viscous term $- 2\epsilon \Delta u_t$. 

Actually, the corresponding abstract problem (but with different nonlinear terms) for \eqref{Generalized_Viscous_Boussinesq_Eq} has been studied by \cite{Biler=1989,Biler=1991}, where global (in time)  well-posedness and time-decay estimates were derived. Facing exactly with the Cauchy problem for \eqref{Generalized_Viscous_Boussinesq_Eq}, the authors in \cite{Varlamov=1996,Varlamov=1996-02} proved local (in time) well-posedness by employing successive approximations. Later, some global (in time)  existence  and blow-up results for \eqref{Generalized_Viscous_Boussinesq_Eq} in 1D were established in \cite{Xu-Luo-Shen-Huang=2017}. The authors in \cite{Wang=2013,Liu-Wang=2014} derived time-decay estimates for solutions to the linearized viscous Boussinesq equation by Green's function, and demonstrated global (in time)  existence of solution with additionally $L^1$ small data. Particularly, under some polynomially weighted $L^1$ assumptions on initial data, asymptotic profiles of solutions have been found in \cite{Ikehata-Soga}. Recently, by ignoring the dissipative part, the authors in \cite{Liu-Wang=2019} completed uniform estimates with respect to the viscosity $\epsilon$ and rigorously demonstrated local (in time) inviscid limits for \eqref{Generalized_Viscous_Boussinesq_Eq} as $\epsilon\downarrow0$. Simultaneously, with the aid of the Littlewood-Paley theory, in the more recent paper \cite{Liu-Wang=2020} they not only have developed some new estimates for a class of dissipative-dispersive linear semigroup but also have proved global (in time) existence of small data Sobolev solutions. 

To the best of the authors' knowledge, although $L^{q'}-L^q$ estimates ($q'$ is the H\"older conjugate of $q$ so that $q\geqslant 2$) for solutions to the corresponding linear Cauchy problem for \eqref{Generalized_Viscous_Boussinesq_Eq} have been deeply studied in \cite{Liu-Wang=2020}, the investigation of the nonlinear viscous Boussinesq equation \eqref{Generalized_Viscous_Boussinesq_Eq} in general $L^q$ spaces, especially for $1\leqslant q<2$, has never been considered in this literature in term of the study of the initial value problem. We will partly answer this question in the present work. 

\subsection{Main purposes of the paper}
In the present paper, we will consider the Cauchy problem for the following viscous Boussinesq equation:
\begin{equation}\label{Equation_Main}
	\begin{cases}
		u_{tt}- \Delta u+ \Delta^2 u- 2\epsilon \Delta u_t= \Delta f(u;p), &x\in \R^n,\ t>0, \\
		u(0,x)= u_0(x),\  u_t(0,x)= u_1(x), &x \in \R^n,
	\end{cases}
\end{equation}
where the smooth nonlinear function behaves $f(u;p)=\mathcal{O}(|u|^p)$ when $|u|\downarrow 0$ equipping $p>1$ exactly as we introduced in foregoing paragraphs, and we are going to study its corresponding linearized Cauchy problem with the vanishing right-hand side as follows:
\begin{equation}\label{Equation_Main_Linear}
	\begin{cases}
		v_{tt}- \Delta v+ \Delta^2 v- 2\epsilon \Delta v_t= 0, &x\in \R^n,\ t>0, \\
		v(0,x)= v_0(x),\  v_t(0,x)= v_1(x), &x \in \R^n,
	\end{cases}
\end{equation}
where $\epsilon\in(0,1)$ stands for the viscous coefficient.

Our first purpose in this paper is to understand qualitative properties of solutions to the linearized viscous Boussinesq equation \eqref{Equation_Main_Linear} by employing the Fourier analysis in the $L^q$ framework, which not only helps us to comprehend the underlying physical phenomena in a certain condition but also plays a pivotal role in developing the corresponding nonlinear problem \eqref{Equation_Main}. Precisely in Section \ref{Section_Lm-Lq}, by applying the theory of modified Bessel's function associated with Fa\`{a} di Bruno's formula (see Appendix \ref{Section_Appendix_A}), we derive $L^m-L^q$ estimates for solutions with any $1\leqslant m\leqslant q\leqslant \infty$, which improve the previous results and enlarge the admissible range of index $q$ (in the $L^q$ solution space). Especially, due to invalidation of some embeddings in $L^q$ spaces if $1\leqslant q<2$, we have to control the Fourier multiplier 
\begin{align}\label{Multiplier}
	\ml{F}^{-1}_{\xi\to x}\left(\mathrm{e}^{-\epsilon|\xi|^2t}\frac{\sin\left(|\xi|\sqrt{(1-\epsilon^2)|\xi|^2+1}\,t\right)}{|\xi|\sqrt{(1-\epsilon^2)|\xi|^2+1}}\right)(t,\cdot)
\end{align}
in several norms by some suitable ways. Roughly speaking, another challenge comes from unclear combined influence from the oscillating structure, dissipative effect and singularities (for low frequencies) of the multiplier \eqref{Multiplier}. Differently from the treatment of dissipative-dispersive estimates stated in \cite{Liu-Wang=2020}, we express the dispersive part by some oscillating functions and deal with them by employing refined WKB analysis.

Furthermore, by naturally taking into account the limit case with the vanishing viscosity in the Cauchy problem \eqref{Equation_Main_Linear}, it turns out to be the inviscid Boussinesq equation
\begin{equation}\label{Equation_Inviscid_Linear}
	\begin{cases}
		v^{0}_{tt}- \Delta v^{0}+ \Delta^2 v^0= 0, &x\in \R^n,\ t>0, \\
		v^0(0,x)= v^0_0(x),\  v^0_t(0,x)= v^0_1(x), &x \in \R^n.
	\end{cases}
\end{equation}
The model signifies that viscous forces are neglected in the propagation of long waves on the surface of shallow water. Formally, equipped identical initial data in viscous and inviscid models, \eqref{Equation_Inviscid_Linear} seems to be the limit model for \eqref{Equation_Main_Linear} with the vanishing viscosity, but in this limit procedure the exponential dissipation will fade away as we mentioned in the above. In Section \ref{Section_Inviscid_Limits}, under the consistent assumption for initial data, we rigorously justify inviscid limits from the viscous case \eqref{Equation_Main_Linear} to the inviscid case \eqref{Equation_Inviscid_Linear} as the viscous coefficient tends to zero, namely, $v\to v^0$ in $L^{\infty}([0,T]\times\mb{R}^n)$ as $\epsilon\downarrow0$ with the rate of convergence $\epsilon$. A further consideration is to prove rigorously the second-order asymptotic expansion, in other words, $v\to v^0+\epsilon\tilde{v}$ in $L^{\infty}([0,T]\times\mb{R}^n)$ as $\epsilon\downarrow0$ with the improved rate of convergence $\epsilon^2$, where $\tilde{v}$ is a solution to the inhomogeneous inviscid Boussinesq equation (to be shown in Subsection \ref{Sub_Sec_Second-Order}). Our approaches are based on the WKB analysis associated with the multi-scale analysis, and energy methods in the Fourier space.

Our second purpose in this paper is to study global (in time) well-posedness and some $L^q$ estimates for solutions to the nonlinear viscous Boussinesq equation \eqref{Equation_Main} with small data. In Section \ref{Section_GESDS}, under some conditions for the power $p$ in the nonlinearity, we construct a time-weighted Banach space in the $L^q$ framework and demonstrate a contraction mapping in this space. Then, the global (in time) existence of unique solution
    $$u\in \ml{C}\left([0,\infty),H^s_q\right)$$
to the Cauchy problem \eqref{Equation_Main} will be proved. This result extends the previous results in the $L^2$ setting to the general $L^q$ setting with $q\in(1,\infty)$.

\color{black}

\medskip

\noindent\textbf{Notation:}  
To end this section, we give some notations to be used throughout this paper. Later, $c$ and $C$ denote some positive constants, which may be changed from line to line. We denote that $f\lesssim g$ if there exists a positive constant $C$ such that $f\leqslant Cg$ and, analogously, for $f\gtrsim g$. We take $\delta_0$ to be the Dirac distribution at $x=0$. Moreover, $\dot{H}^s_q$ with $s\geqslant0$ and $1\leqslant q\leqslant \infty$, denote Riesz potential spaces based on $L^q$. Furthermore, $|D|^s$ and $\langle D\rangle^s$ with $s\geqslant0$ stand for the fractional Laplacian and the pseudo-differential operators with the symbol $|\xi|^s$ and $\langle\xi\rangle^s$, respectively. Here, the Japanese bracket is denoted by $\langle\xi\rangle:=\sqrt{1+|\xi|^2}$.

\color{black}
\section{Some $L^m-L^q$ estimates for solutions}\label{Section_Lm-Lq}
\subsection{Pretreatments in the Fourier space}
To begin with our analysis in the $L^q$ space, let us apply the partial Fourier transform with respect to spatial variables $x$ to the linear problem \eqref{Equation_Main_Linear} such that $\hat{v}(t,\xi):=\ml{F}_{x\to\xi}(v(t,x))$. Then, we arrive at
\begin{equation}\label{Equation_Main_Linear_Fourier}
	\begin{cases}
		\hat{v}_{tt}+ 2\epsilon |\xi|^2\hat{v}_t+|\xi|^2(1+|\xi|^2)\hat{v}= 0, &\xi\in \R^n,\ t>0, \\
		\hat{v}(0,\xi)= \hat{v}_0(\xi),\  \hat{v}_t(0,\xi)= \hat{v}_1(\xi), &\xi \in \R^n,
	\end{cases}
\end{equation}
whose characteristic roots $\lambda_{\pm}=\lambda_{\pm}(|\xi|)$ are given by
\begin{align}\label{Characteristic_ROOT}
	\lambda_{\pm}(|\xi|)=-\epsilon|\xi|^2\pm i|\xi|\sqrt{(1-\epsilon^2)|\xi|^2+1}
\end{align}
thanks to $\epsilon\in(0,1)$.
\begin{remark}
 According to the expression \eqref{Characteristic_ROOT}, dissipative structure as well as dispersive structure exist simultaneously. As a consequence, it is delicate and non-trivial to derive the decay properties contributed by them.
\end{remark}
 The pairwise distinct characteristic roots \eqref{Characteristic_ROOT} allow us to represent the solution to \eqref{Equation_Main_Linear_Fourier} by
\begin{align*}
	\hat{v}(t,\xi)&=\underbrace{\frac{\lambda_+(|\xi|)\mathrm{e}^{\lambda_-(|\xi|)t}-\lambda_-(|\xi|)\mathrm{e}^{\lambda_+(|\xi|)t}}{\lambda_+(|\xi|)-\lambda_-(|\xi|)}}_{=:\widehat{K}_0(t,|\xi|)}\hat{v}_0(\xi)+\underbrace{\frac{\mathrm{e}^{\lambda_+(|\xi|)t}-\mathrm{e}^{\lambda_-(|\xi|)t}}{\lambda_+(|\xi|)-\lambda_-(|\xi|)}}_{=:\widehat{K}_1(t,|\xi|)}\hat{v}_1(\xi).
\end{align*}
Our assignment in this part is to process estimates for the Fourier multipliers 
\begin{align*}
	\ml{F}^{-1}_{\xi\to x}(\widehat{K}_0(t,|\xi|)) \ \ \mbox{and}\ \ \ml{F}^{-1}_{\xi\to x}(\widehat{K}_1(t,|\xi|))
\end{align*}
 as a whole, respectively, because we expect some cancellations between two dispersive parts in $\mathrm{e}^{\lambda_{\pm}(|\xi|)t}$. Precisely, the kernels can be rewritten by
\begin{align}
	\widehat{K}_0(t,|\xi|)&=\mathrm{e}^{-\epsilon|\xi|^2t}\left(\cos\left(|\xi|\sqrt{(1-\epsilon^2)|\xi|^2+1}\,t\right)+\frac{\epsilon|\xi|\sin\left(|\xi|\sqrt{(1-\epsilon^2)|\xi|^2+1}\,t\right)}{\sqrt{(1-\epsilon^2)|\xi|^2+1}}\right), \label{K_0-Formula} \\
	\widehat{K}_1(t,|\xi|)&=\mathrm{e}^{-\epsilon|\xi|^2t}\frac{\sin\left(|\xi|\sqrt{(1-\epsilon^2)|\xi|^2+1}\,t\right)}{|\xi|\sqrt{(1-\epsilon^2)|\xi|^2+1}} \label{K_1-Formula}.
\end{align}
Noticing that since the dissipative part $\mathrm{e}^{-\epsilon|\xi|^2t}$, the oscillating part from the     sine function as well as the singular part $|\xi|^{-1}$ (for low frequencies) exist concurrently in the Fourier multiplier \eqref{K_1-Formula}, it is significant to see their combined influence. 
\begin{remark}
Here, we give some comments for our methods (will be utilized in the next parts) to estimate the Fourier multiplier corresponding to \eqref{K_1-Formula} as follows:
\begin{itemize}
	\item Differently from the study of the viscous Boussinesq equation \eqref{Equation_Main_Linear} in the $L^2$ framework \cite{Ikehata-Soga}, the Plancherel theorem does not work anymore in the general $L^r$ framework for any $r\in [1,\ity]$.
	\item Differently from the study of the viscous Boussinesq equation \eqref{Equation_Main_Linear} in the $L^q$ framework with $q\geqslant 2$ (see \cite{Liu-Wang=2020}), the embedding from the $L^q$ physical space to the $L^{q'}$ phase space (with $1/q+1/q'=1$ and $q\geqslant 2$) does not work anymore in the general $L^r$ framework if $1\leqslant r<2$.
\end{itemize}
For the above reasons, the previous used approaches bring some difficulties to analyze finely decay properties of $\|K_1(t,|D|)\|_{L^r}$ for any $r\in [1,\ity]$. 
\end{remark}

By considering the cases of low and high frequencies individually, one gets
\begin{itemize}
	\item $\lambda_{\pm} \sim - \epsilon|\xi|^{2}\pm i|\xi|$, $\lambda_+ -\lambda_- \sim i|\xi|$ for low frequencies $|\xi|\ll 1$;
	\item $\lambda_{\pm} \sim -\epsilon|\xi|^{2}\pm i\sqrt{1-\epsilon^2}|\xi|^2$, $\lambda_+ -\lambda_- \sim i|\xi|^2$ for high frequencies $|\xi|\gg1$.
\end{itemize}
We now introduce the radial, smooth cut-off functions $\chi_{\rm L}(|\xi|)$, $\chi_{\rm H}(|\xi|)$ and $\chi_{\rm M}(|\xi|)$ defined by 
\begin{gather*}
	\chi_{\rm L} (|\xi|) := \begin{cases}
		1 &\mbox{if}\ \ |\xi| \leqslant (1-\epsilon^2)^2/2, \\
		0 &\mbox{if}\ \ |\xi| \geqslant(1-\epsilon^2)^2, 
	\end{cases} \qquad
	\chi_{\rm H} (|\xi|) := \begin{cases}
		1 &\mbox{if }\ \ |\xi| \geqslant4(1-\epsilon^2)^{-2}, \\
		0 &\mbox{if }\ \ |\xi| \leqslant2(1-\epsilon^2)^{-2}, 
	\end{cases}  
\end{gather*}
and $\chi_{\rm M} (|\xi|) := 1- \chi_{\rm L} (|\xi|) - \chi_{\rm H} (|\xi|)$, respectively, for decomposing solutions to the Cauchy problem \eqref{Equation_Main_Linear} into three parts localized separately to low, middle and high frequency zones. To be specific, we rewrite the solution by
\begin{align*}
	v(t,x)
	=\mathcal{F}^{-1}_{\xi\to x}\big(\hat{v}(t,\xi)\chi_{\rm L}(|\xi|)\big)(t,x)+\mathcal{F}^{-1}_{\xi\to x}\big(\hat{v}(t,\xi)\chi_{\rm M}(|\xi|)\big)(t,x)+\mathcal{F}^{-1}_{\xi\to x}\big(\hat{v}(t,\xi)\chi_{\rm H}(|\xi|)\big)(t,x).
\end{align*}
In order to obtain $L^m-L^q$ estimates for solutions to \eqref{Equation_Main_Linear} with $1\leqslant m\leqslant q\leqslant \ity$, by the mean of Young's inequality, $L^r$ estimates for the kernels $K_0(t,|D|)$ and $K_1(t,|D|)$ with any $r\in[1,\infty]$ play essential roles. Consequently, we will establish both $L^1$ estimates as well as $L^\ity$ estimates for these kernels, and employ some interpolation theorems to conclude our desired estimates.
\begin{remark}
	Because the real parts of the characteristic roots $\lambda_{\pm}$ are negative in the middle frequency zone 
	\begin{align*}
		\left\{\xi\in \R^n : |\xi| \in \left[(1-\epsilon^2)^2/2,4(1-\epsilon^2)^{-2}\right]\right\},
	\end{align*}
the corresponding estimates for this zone possess an exponential decay
\begin{align}\label{estimate_middle}
 \left\|\ml{F}^{-1}_{\xi\to x}\left(|\xi|^{\alpha}\widehat{K}_{j}(t,|\xi|)\chi_{\mathrm{M}}(|\xi|)\right)(t,\cdot)\right\|_{L^r}\lesssim \mathrm{e}^{-ct}
\end{align}
with $j=0,1$, for $\alpha\geqslant 0$ and any $r\in[1,\infty]$, where $c$ is a positive constant. For this reason, it suffices to devote our consideration to the low frequency zone and the high frequency zone.
\end{remark}

\subsection{Estimates for high frequencies}
First of all, let us derive $L^1$ estimate for the following oscillating integral in the high frequency zone:
\begin{align*}
	\ml{M}(t,x):=\mathcal{F}^{-1}_{\xi\to x}\left(\mathrm{e}^{-c_1 |\xi|^2 t}|\xi|^{\alpha}\frac{\sin \left(c_2 |\xi|^2 g_{\rm H}(|\xi|) t\right)}{g_{\rm H}(|\xi|)}\chi_{\rm H}(|\xi|)\right)(t,x),
\end{align*}
with $\alpha \geqslant 0$, where $c_1$ is a positive constant, $c_2 \ne 0$ is a real constant and $g_{\rm H}(|\xi|):= \sqrt{(1-\epsilon^2)+|\xi|^{-2}}$.

\begin{lemma} \label{Key_Lemma}
Let $\alpha\geqslant0$.	The following estimates hold in $\R^n$ for any $n \geqslant 1$:
	\begin{align*}
	    \|\ml{M}(t,\cdot)\|_{L^1}\lesssim t^{-\frac{\alpha}{2}}.
	\end{align*}
\end{lemma}
\begin{proof} Our strategies are based on the ideas from Lemma 5 in \cite{Nara-Reissig=2013}. Carrying out the change of variable $\xi=t^{-\frac{1}{2}}\eta$ we have
	\begin{align*}
		\ml{M}(t,x)= t^{-\frac{n+\alpha}{2}}\ml{F}_{\eta\to x}^{-1}\left(\mathrm{e}^{-c_1 |\eta |^2}|\eta|^{\alpha}\frac{\sin\left(c_2 |\eta|^2 g_{\rm H}\left(t^{-\frac{1}{2}}|\eta|\right)\right)}{g_{\rm H}\left(t^{-\frac{1}{2}}|\eta|\right)}\chi_{\rm H}\left(t^{-\frac{1}{2}}|\eta|\right)\right)(t,t^{-\frac{1}{2}}x).
	\end{align*}
	By letting $t^{-\frac{1}{2}}x\mapsto x$, it follows that
	\begin{align*}
		\|\ml{M}(t,\cdot)\|_{L^1} = t^{-\frac{\alpha}{2}}\left\|\ml{F}_{\eta\to x}^{-1}\left(\mathrm{e}^{-c_1 |\eta |^2}|\eta|^{\alpha}\frac{\sin\left(c_2 |\eta|^2 g_{\rm H}\left(t^{-\frac{1}{2}}|\eta|\right)\right)}{g_{\rm H}\left(t^{-\frac{1}{2}}|\eta|\right)}\chi_{\rm H}\left(t^{-\frac{1}{2}}|\eta|\right)\right)(t,\cdot)\right\|_{L^1}.
	\end{align*}
	Hence, in the sequel let us devote our attention to estimate the Fourier multiplier
	\begin{equation}\label{H-FourierMultiplier}
		\mathcal{H}(t,x):= \ml{F}_{\eta\to x}^{-1}\left(\mathrm{e}^{-c_1 |\eta |^2}|\eta|^{\alpha}\frac{\sin\left(c_2 |\eta|^2 g_{\rm H}\left(t^{-\frac{1}{2}}|\eta|\right)\right)}{g_{\rm H}\left(t^{-\frac{1}{2}}|\eta|\right)}\chi_{\rm H}\left(t^{-\frac{1}{2}}|\eta|\right)\right)(t,x).
	\end{equation}
	At the first stage, we restrict ourselves into  $|x| \leqslant 1$. It deduces immediately
	\begin{align*}
		\|\mathcal{H}(t,\cdot)\|_{L^1(|x| \leqslant 1)}& =\int_{|x|\leqslant 1}\left|\int_{|\eta|>\frac{2}{(1-\epsilon^2)^2}t^{\frac{1}{2}}}\mathrm{e}^{ix\cdot\eta-c_1|\eta|^2}|\eta|^{\alpha}\frac{\sin\left(c_2 |\eta|^2 g_{\rm H}\left(t^{-\frac{1}{2}}|\eta|\right)\right)}{g_{\rm H}\left(t^{-\frac{1}{2}}|\eta|\right)}\mathrm{d}\eta\right|\mathrm{d}x\\
		 &\lesssim \int_{|\eta|>\frac{2}{(1-\epsilon^2)^2}t^{\frac{1}{2}}}|\eta|^{\alpha}\mathrm{e}^{-c_1|\eta|^2}\mathrm{d}\eta\lesssim 1,
	\end{align*}
    in which we used $g_{\mathrm{H}}(t^{-\frac{1}{2}}|\eta|)\geqslant\sqrt{1-\epsilon^2}$. Therefore, we may claim the following estimate:
    \begin{align*}
	    \|\ml{M}(t,\cdot)\|_{L^1(|x| \leqslant 1)} \lesssim t^{-\frac{\alpha}{2}}\|\mathcal{H}(t,\cdot)\|_{L^1(|x| \leqslant 1)}\lesssim t^{-\frac{\alpha}{2}}.
    \end{align*}
	In the next stage, we consider the case with $|x| \geqslant 1$. Due to the radially symmetric property of the functions in the parenthesis of \eqref{H-FourierMultiplier} with respect to $\eta$, the multiplier $\mathcal{H}(t,x)$ is also radially symmetric with respect to $x$. The employment of the modified Bessel's functions from Proposition \ref{Modified-Bessel-Funct} gives
	\begin{equation}
		\mathcal{H}(t,x)=c \int_0^\ity \mathrm{e}^{-c_1 r^2}r^{\alpha+n-1}\frac{\sin\left(c_2 r^2 g_{\rm H}\left(t^{-\frac{1}{2}}r\right)\right)}{g_{\rm H}\left(t^{-\frac{1}{2}}r\right)}\chi_{\rm H}\left(t^{-\frac{1}{2}}r\right)\tilde{\mathcal{J}}_{\frac{n}{2}-1}(r|x|) {\rm d}r. \label{l3.1.9}
	\end{equation}
    Because the value of $\tilde{\ml{J}}_{\frac{n}{2}-1}(s)$ for minimal dimensions for the odd case (that is $n=3$) and the even case (that is $n=2$) are different, i.e. the properties (3) and (4) in Proposition \ref{Modified-Bessel-Funct}, we will divide our discussion into two circumstances.
    \begin{itemize}
	\item[$\bullet$] \textit{\textbf{Case 1:} Let us consider odd dimensions $n=2m+1$ with $m \geqslant 1$.} The one-dimensional case will be explained in Remark \ref{Rem_n=1}. We introduce the vector field 
	\begin{align*}
			\mathcal{V}(h(r)):= \f{\rm d}{{\rm d}r}\left(\frac{1}{r}h(r)\right)
	\end{align*}
	 as those in \cite{Nara-Reissig=2013} and then perform $m+1$ steps of integration by parts to express \eqref{l3.1.9} in the form
	\begin{align}
		\mathcal{H}(t,x)&= \frac{c}{|x|^n} \int_0^\ity \partial_r \left(\mathcal{V}^{(m)} \left( \mathrm{e}^{-c_1 r^2}\frac{\sin\left(c_2 r^2 g_{\rm H}\left(t^{-\frac{1}{2}}r\right)\right)}{g_{\rm H}\left(t^{-\frac{1}{2}}r\right)}\chi_{\rm H}\left(t^{-\frac{1}{2}}r\right) r^{\alpha+2m}\right)\right) \sin(r|x|){\rm d}r \notag \\
		&=\sum_{\substack{1\leqslant j+k \leqslant m+1\\ j,\,k \geqslant 0}} \frac{c_{jk}}{|x|^n} \int_0^\ity \partial_r^{j+1} \mathrm{e}^{-c_1 r^2} \partial^k_r \left(\frac{\sin\left(c_2 r^2 g_{\rm H}\left(t^{-\frac{1}{2}}r\right)\right)}{g_{\rm H}\left(t^{-\frac{1}{2}}r\right)}\chi_{\rm H}\left(t^{-\frac{1}{2}}r\right)\right) r^{j+k+\alpha} \sin(r|x|){\rm d}r \label{l3.1.10}
	\end{align}
	with some suitable constants $c_{jk}>0$.	To control the function $\mathcal{H}(t,x)$, at first we need to affirm the following auxiliary estimates on the support of $\chi_{\rm H}\left(t^{-\frac{1}{2}}r\right)$ and on the support of its partial derivatives with respect to $r$:
	\begin{align}
		\left|\partial_r^j \mathrm{e}^{-c_1 r^2}\right|
		&\lesssim
		\begin{cases}
			\mathrm{e}^{-c_1 r^2} &\text{if}\ \ j =0, \\
			\mathrm{e}^{-c_1 r^2}r^{2-j}(1+r^2)^{j-1} &\text{if}\ \ j=1,\dots,m,
		\end{cases} \notag\\
		\left|\partial^k_r \left(\frac{\sin\left(c_2 r^2 g_{\rm H}\left(t^{-\frac{1}{2}}r\right)\right)}{g_{\rm H}\left(t^{-\frac{1}{2}}r\right)}\chi_{\rm H}\left(t^{-\frac{1}{2}}r\right)\right)\right|
		&\lesssim  \begin{cases}
			\min\{1,r^2\} &\text{if}\ \ k =0, \\
			r^{2-k} (1+r^2)^{k-1} &\text{if}\ \ k=1,\dots,m.
		\end{cases}\label{Est_sup}
	\end{align}
	Indeed, it is obvious to show the first three estimates in the above by straightforward computations. For the remaining one, we mainly divide our verification into two steps by  employing Fa\`{a} di Bruno's formula (see Proposition \ref{FadiBruno'sformula1}) in two components. To be specific, we may derive it by two steps.
		
	\noindent \underline{Step 1}: The application of Proposition \ref{FadiBruno'sformula1} with $h(s)=\sqrt{s}$ and $\psi(r)=1-\epsilon^2+tr^{-2}$ leads to
	\begin{align*}
		\left|\partial_r^k g_{\rm H}\left(t^{-\frac{1}{2}}r\right)\right| &\lesssim \left| \sum_{\substack{1\cdot m_1+\cdots+k\cdot m_k=k,\, m_i \geqslant 0}} \psi(r)^{\frac{1}{2}-(m_1+\cdots+m_k)}\prod_{j=1}^k \left(tr^{-2-j}\right)^{m_j}\right|\\
		&\lesssim \sum_{\substack{1\cdot m_1+\cdots+k\cdot m_k=k,\, m_i \geqslant 0}} \left(tr^{-2}\right)^{m_1+\cdots+m_k} r^{-k}\lesssim r^{-k},
	\end{align*}
	where we used $1-\epsilon^2 \leqslant \psi(r) \leqslant C$  and $tr^{-2} \leqslant C$ for $r \gg t^{\frac{1}{2}}$ in the last two chains, respectively. Moreover, by the same procedure one achieves
	\begin{align}
		\left| \partial_r^k \left(\frac{1}{g_{\rm H}\left(t^{-\frac{1}{2}}r\right)} \right)\right|  \lesssim r^{-k} \ \ \text{for}\ \ k=1,\dots,m. \label{l3.3.3}
	\end{align}
		
	\noindent \underline{Step 2}: The use of Proposition \ref{FadiBruno'sformula1} with $h(s)=\sin(c_2s)$ and $\psi(r)=r^2 g_{\rm H}\left(t^{-\frac{1}{2}}r\right)$ yields that
	\begin{align*}
		&\left|\partial_r^k \sin\left(c_2 r^2 g_{\rm H}\left(t^{-\frac{1}{2}}r\right)\right)\right| \\
		&\qquad \lesssim \left|\sum_{\substack{1\cdot m_1+\cdots+k\cdot m_k=k,\, m_i \geqslant 0}} \left(\sin\left(c_2 r^2 g_{\rm H}\left(t^{-\frac{1}{2}}r\right)\right)\right)^{(m_1+\cdots+m_k)} \,\, \prod_{j=1}^k \left(\partial_r^j \left(r^2 g_{\rm H}\left(t^{-\frac{1}{2}}r\right)\right)\right)^{m_j} \right| \\
		&\qquad \lesssim \left| \sum_{\substack{1\cdot m_1+\cdots+k\cdot m_k=k,\, m_i \geqslant 0}} \,\, \prod_{j=1}^k \left(\sum_{\ell=0}^j  r^{2-j+\ell} \partial_r^\ell g_{\rm H}\left(t^{-\frac{1}{2}}r\right) \right)^{m_j} \right|,
	\end{align*}
	and it follows
	\begin{align}
        \left|\partial_r^k \sin\left(c_2 r^2 g_{\rm H}\left(t^{-\frac{1}{2}}r\right)\right)\right|			& \lesssim \sum_{\substack{1\cdot m_1+\cdots+k\cdot m_k=k,\, m_i \geqslant 0}} \,\, \prod_{j=1}^k (r^{2-j})^{m_j }\notag\\
        & \lesssim r^{2-k} \sum_{\substack{1\cdot m_1+\cdots+k\cdot m_k=k,\, m_i \geqslant 0}} r^{2(m_1+\cdots+m_k-1)} \notag \\
			& \lesssim r^{2-k} \big(1+ r^{2(k-1)}\big) \lesssim r^{2-k} \big(1+r^2\big)^{k-1}. \label{l3.3.4}
	\end{align}
	Then, we use the product rule for higher-order derivatives linked to \eqref{l3.3.3} and \eqref{l3.3.4} to conclude
	\begin{align*}
		\left| \partial_r^k \left(\frac{\sin\left(c_2 r^2 g_{\rm H}\left(t^{-\frac{1}{2}}r\right)\right)}{g_{\rm H}\left(t^{-\frac{1}{2}}r\right)}\right) \right| \lesssim r^{2-k}(1+r^2)^{k-1} \ \  \text{for}\ \  k=1,\dots,m.
	\end{align*}
	As a consequence, we can see that the last estimate verifies the assertion \eqref{Est_sup} for $k=1,\dots,m$.
		
	From the derived auxiliary estimates, we obtain
	\begin{align*}
		&\left|\partial_r^{j+1} \mathrm{e}^{-c_1 r^2}\partial^k_r \left(\frac{\sin\left(c_2 r^2 g_{\rm H}\left(t^{-\frac{1}{2}}r\right)\right)}{g_{\rm H}\left(t^{-\frac{1}{2}}r\right)}\chi_{\rm H}\left(t^{-\frac{1}{2}}r\right)\right) r^{j+k+\alpha}\right| \\
		&\qquad \lesssim
		\begin{cases}
			\mathrm{e}^{-c_1 r^2} r^{\alpha+1}(1+r^2)^j &\text{if} \ \ k =0, \\
			\mathrm{e}^{-c_1 r^2} r^{\alpha+3}(1+r^2)^{k+j-1} &\text{if}\ \ k=1,\dots,m.
		\end{cases}
	\end{align*}
	
	Let us now come back to the aimed estimate for $\|\ml{H}(t,\cdot)\|_{L^1(|x|\geqslant 1)}$. For $k=0$, we divide the integral in \eqref{l3.1.10} into two sub-integrals from $t^{\frac{1}{2}}$ to $\pi/(2|x|)$ and from $\pi/(2|x|)$ to $\ity$.	We remark that the support of $\chi_{\mathrm{H}}(t^{-\frac{1}{2}}r)$ localizes in $\{(t,r)\in\mb{R}^2_+: t^{-\frac{1}{2}}r>2(1-\epsilon^2)^{-2} \}$ so that $t^{\frac{1}{2}}<\pi/(2|x|)$ always holds for $r<\pi/(2|x|)$. For one thing, one has
	\begin{align}
		&\left|\int_{t^{\frac{1}{2}}}^{\frac{\pi}{2|x|}} \partial_r^{j+1} \mathrm{e}^{-c_1 r^2} \frac{\sin\left(c_2 r^2 g_{\rm H}\left(t^{-\frac{1}{2}}r\right)\right)}{g_{\rm H}\left(t^{-\frac{1}{2}}r\right)}\chi_{\rm H}\left(t^{-\frac{1}{2}}r\right)r^{j+\alpha} \sin(r|x|){\rm d}r \right| \notag \\
		&\qquad \lesssim \int_{t^{\frac{1}{2}}}^{\frac{\pi}{2|x|}} r^{\alpha+1}(1+r^2)^j {\rm d}r \lesssim \frac{1}{|x|^{\alpha+2}}+ \frac{1}{|x|^{\alpha+2(j+1)}} \lesssim \frac{1}{|x|^{\alpha+2}}. \label{l3.1.11}
	\end{align}
	For another, performing one more step of integration by parts we have
	\begin{align*}
		\widetilde{H}(t,x)&:=\left|\int_{\frac{\pi}{2|x|}}^\ity \partial_r^{j+1} \mathrm{e}^{-c_1 r^2} \frac{\sin\left(c_2 r^2 g_{\rm H}\left(t^{-\frac{1}{2}}r\right)\right)}{g_{\rm H}\left(t^{-\frac{1}{2}}r\right)}\chi_{\rm H}\left(t^{-\frac{1}{2}}r\right)r^{j+\alpha} \sin(r|x|){\rm d}r \right| \\
		& \,\,\lesssim \frac{1}{|x|}\left| \partial_r^{j+1} \mathrm{e}^{-c_1 r^2} \frac{\sin\left(c_2 r^2 g_{\rm H}\left(t^{-\frac{1}{2}}r\right)\right)}{g_{\rm H}\left(t^{-\frac{1}{2}}r\right)}\chi_{\rm H}\left(t^{-\frac{1}{2}}r\right)r^{j+\alpha} \cos(r|x|) \right |_{r=\frac{\pi}{2|x|}}^{r=\ity} \\
		&\,\,\quad + \frac{1}{|x|}\int_{\frac{\pi}{2|x|}}^\ity \left|\partial_r \left(\partial_r^{j+1} \mathrm{e}^{-c_1 r^2} \frac{\sin\left(c_2 r^2 g_{\rm H}\left(t^{-\frac{1}{2}}r\right)\right)}{g_{\rm H}\left(t^{-\frac{1}{2}}r\right)}\chi_{\rm H}\left(t^{-\frac{1}{2}}r\right)r^{j+\alpha}\right) \cos(r|x|)\right| {\rm d}r, 
	\end{align*}
	and then
	\begin{align}
		\widetilde{H}(t,x)& \lesssim \frac{1}{|x|}\int_{\frac{\pi}{2|x|}}^\ity \mathrm{e}^{-c_1 r^2} r^{\alpha}(1+r^2)^{j+1} {\rm d}r \notag \\
		& \lesssim
		\frac{1}{|x|} \left( \int_{\frac{\pi}{2|x|}}^{\frac{\pi}{2}} r^{\alpha} {\rm d}r + \int_{\frac{\pi}{2}}^\ity \mathrm{e}^{-c_1 r^2} r^{2(j+1)+\alpha} {\rm d}r \right) \lesssim \frac{1}{|x|}, \label{l3.1.12}
	\end{align}
	where we noticed that
	\begin{align*}
		\left|\partial_r \left(\partial_r^{j+1} \mathrm{e}^{-c_1 r^2} \frac{\sin\left(c_2 r^2 g_{\rm H}\left(t^{-\frac{1}{2}}r\right)\right)}{g_{\rm H}\left(t^{-\frac{1}{2}}r\right)}\chi_{\rm H}\left(t^{-\frac{1}{2}}r\right)r^{j+\alpha}\right)\right| \lesssim \mathrm{e}^{-c_1 r^2} r^{\alpha}(1+r^2)^{j+1}.
		\end{align*}
	Regarding $k=1,\dots,m$, we may treat in an analogous procedure as we did for $k=0$ to claim the next estimates:
	\begin{equation} \label{l3.1.13}
		\left|\int_{t^{\frac{1}{2}}}^{\frac{\pi}{2|x|}} \partial_r^{j+1} \mathrm{e}^{-c_1 r^2} \partial^k_r \left(\frac{\sin\left(c_2 r^2 g_{\rm H}\left(t^{-\frac{1}{2}}r\right)\right)}{g_{\rm H}\left(t^{-\frac{1}{2}}r\right)}\chi_{\rm H}\left(t^{-\frac{1}{2}}r\right)\right) r^{j+k+\alpha} \sin(r|x|){\rm d}r\right| \lesssim \frac{1}{|x|^{\alpha+4}}, 
	\end{equation}
	together with
	\begin{equation} \label{l3.1.14}
		\left|\int_{\frac{\pi}{2|x|}}^\ity \partial_r^{j+1} \mathrm{e}^{-c_1 r^2} \partial^k_r \left(\frac{\sin\left(c_2 r^2 g_{\rm H}\left(t^{-\frac{1}{2}}r\right)\right)}{g_{\rm H}\left(t^{-\frac{1}{2}}r\right)}\chi_{\rm H}\left(t^{-\frac{1}{2}}r\right)\right) r^{j+k+\alpha} \sin(r|x|){\rm d}r \right| \lesssim \frac{1}{|x|},
	\end{equation}
	where we also noted that
	\begin{align*}
		\left|\partial_r \left( \partial_r^{j+1} \mathrm{e}^{-c_1 r^2} \partial^k_r \left(\frac{\sin\left(c_2 r^2 g_{\rm H}\left(t^{-\frac{1}{2}}r\right)\right)}{g_{\rm H}\left(t^{-\frac{1}{2}}r\right)}\chi_{\rm H}\left(t^{-\frac{1}{2}}r\right)\right) r^{j+k+\alpha} \sin(r|x|)\right)\right| \lesssim \mathrm{e}^{-c_1 r^2} r^{\alpha+2}(1+r^2)^{k+j}.
	\end{align*}
	For these reasons, collecting the derived estimates from \eqref{l3.1.11} to \eqref{l3.1.14} we have established the terms $|x|^{-(n+\alpha+2)}$, $|x|^{-(n+1)}$ and $|x|^{-(n+\alpha+4)}$ for the estimate of $|\ml{H}(t,x)|$, which guarantee the $L^1$ integrablity with respect to $x$. Therefore, we conclude the estimate for $n=2m+1$ such that
	\begin{align*}
		\|\ml{M}(t,\cdot)\|_{L^1(|x| \geqslant 1)}\lesssim t^{-\frac{\alpha}{2}}.
	\end{align*}
	\item[$\bullet$] \textit{\textbf{Case 2:} Let us consider even dimensions $n=2m$ with $m \geqslant 1$.} We carry out $m-1$ steps of integration by parts to express \eqref{l3.1.9} in the form
	\begin{align*}
		\mathcal{H}(t,x) &= \frac{c}{|x|^{2m-2}} \int_0^\ity \mathcal{V}^{(m-1)} \left( \mathrm{e}^{-c_1 r^2}\frac{\sin\left(c_2 r^2 g_{\rm H}\left(t^{-\frac{1}{2}}r\right)\right)}{g_{\rm H}\left(t^{-\frac{1}{2}}r\right)}\chi_{\rm H}\left(t^{-\frac{1}{2}}r\right) r^{\alpha+2m-1}\right) \tilde{\mathcal{J}}_0 (r|x|){\rm d}r \\
		&=\sum_{j=0}^{m-1}\frac{c_j}{|x|^{2m-2}} \int_0^\ity \partial_r^j \left( \mathrm{e}^{-c_1 r^2}\frac{\sin\left(c_2 r^2 g_{\rm H}\left(t^{-\frac{1}{2}}r\right)\right)}{g_{\rm H}\left(t^{-\frac{1}{2}}r\right)}\chi_{\rm H}\left(t^{-\frac{1}{2}}r\right) r^{\alpha}\right) r^{j+1} \tilde{\mathcal{J}}_0 (r|x|){\rm d}r \\
		&=:\sum_{j=0}^{m-1} c_j \mathcal{I}_j(t,x),
	\end{align*}
    with some suitable constants $c_j>0$. Employing the first rule of the modified Bessel's functions for $\mu=1$ and the fifth rule for $\mu=0$ from Proposition \ref{Modified-Bessel-Funct} and then performing two more steps of integration by parts we arrive at
    \begin{align*}
    	|\mathcal{I}_0(t,x)|= \frac{1}{|x|^n} \int_0^\ity \left|\partial_r \left( \partial_r \left( \mathrm{e}^{-c_1 r^2}\frac{\sin\left(c_2 r^2 g_{\rm H}\left(t^{-\frac{1}{2}}r\right)\right)}{g_{\rm H}\left(t^{-\frac{1}{2}}r\right)}\chi_{\rm H}\left(t^{-\frac{1}{2}}r\right) r^{\alpha}\right) r \right) \tilde{\mathcal{J}}_0 (r|x|)\right| {\rm d}r. 
    \end{align*}
	For $j=1,\dots,m$, one may verify the next relation on the support of $\chi_{\mathrm{H}}(t^{\frac{1}{2}}r)$ and on the support of its derivatives:
	\begin{align*}
		\left| \partial^j_r \left( \mathrm{e}^{-c_1 r^2}\frac{\sin\left(c_2 r^2 g_{\rm H}\left(t^{-\frac{1}{2}}r\right)\right)}{g_{\rm H}\left(t^{-\frac{1}{2}}r\right)}\chi_{\rm H}\left(t^{-\frac{1}{2}}r\right) r^{\alpha}\right)\right| \lesssim \mathrm{e}^{-c_1 r^2} (1+ r^2)^{j-1}r^{\alpha+2-j}.
	\end{align*}
	Consequently, it results
	\begin{align*}
		\left|\partial_r \left( \partial_r \left( \mathrm{e}^{-c_1 r^2}\frac{\sin\left(c_2 r^2 g_{\rm H}\left(t^{-\frac{1}{2}}r\right)\right)}{g_{\rm H}\left(t^{-\frac{1}{2}}r\right)}\chi_{\rm H}\left(t^{-\frac{1}{2}}r\right) r^{\alpha}\right) r \right)\right| \lesssim \mathrm{e}^{-c_1 r^2} (1+ r^2)r^{\alpha+1}.
	\end{align*}
	Due to the property that $|\tilde{\mathcal{J}}_0(s)| \leqslant C$ for $s \in [0,1]$, we get the estimate for $t^{\frac{1}{2}}<1/|x|$
	\begin{align}
		&\int_{t^{\frac{1}{2}}}^{\frac{1}{|x|}} \left|\partial_r \left( \partial_r \left( \mathrm{e}^{-c_1 r^2}\frac{\sin\left(c_2 r^2 g_{\rm H}\left(t^{-\frac{1}{2}}r\right)\right)}{g_{\rm H}\left(t^{-\frac{1}{2}}r\right)}\chi_{\rm H}\left(t^{-\frac{1}{2}}r\right) r^{\alpha}\right) r \right) \tilde{\mathcal{J}}_0(r|x|)\right| {\rm d}r \notag \\
		&\qquad \lesssim \int_{t^{\frac{1}{2}}}^{\frac{1}{|x|}} \mathrm{e}^{-c_1 r^2} r^{\alpha+1}{\rm d}r\lesssim \int_{t^{\frac{1}{2}}}^{\frac{1}{|x|}}r^{\alpha+1}{\rm d}r \lesssim \frac{1}{|x|^{\alpha+2}}. \label{l3.1.16}
	\end{align}
	In addition, because $|\tilde{\mathcal{J}}_0(s)| \leqslant Cs^{-\frac{1}{2}}$ for $s>1$, one finds
	\begin{align}
		&\int_{\frac{1}{|x|}}^\ity \left|\partial_r \left( \partial_r \left( \mathrm{e}^{-c_1 r^2}\frac{\sin\left(c_2 r^2 g_{\rm H}\left(t^{-\frac{1}{2}}r\right)\right)}{g_{\rm H}\left(t^{-\frac{1}{2}}r\right)}\chi_{\rm H}\left(t^{-\frac{1}{2}}r\right) r^{\alpha}\right) r \right) \tilde{\mathcal{J}}_0(r|x|)\right| {\rm d}r \notag \\
		&\qquad\lesssim \frac{1}{|x|^{\frac{1}{2}}} \int_{\frac{1}{|x|}}^\ity \mathrm{e}^{-c_1 r^2} (1+ r^2)r^{\alpha+\frac{1}{2}} {\rm d}r \notag \\
		&\qquad\lesssim \frac{1}{|x|^{\frac{1}{2}}} \left(\int_{\frac{1}{|x|}}^1 r^{\alpha+\frac{1}{2}}{\rm d}r+ \int_1^\ity \mathrm{e}^{-c_1 r^2}r^{\alpha+\frac{5}{2}}{\rm d}r\right) \lesssim \frac{1}{|x|^{\frac{1}{2}}}. \label{l3.1.17}
	\end{align}
	By combining \eqref{l3.1.16} and \eqref{l3.1.17} we have derived the terms $|x|^{-(n+\alpha+2)}$ and $|x|^{-(n+\frac{1}{2})}$ on the estimate of $|\ml{I}_0(t,x)|$, which ensure the $L^1$ integrable property with respect to $x$. For this reason, we are able to deduce that
	\begin{align*}
		\|\mathcal{I}_0(t,\cdot)\|_{L^1(|x|\geqslant 1)} \lesssim 1.
	\end{align*}
	Let $j \in [1,m-1]$ be an integer. Then, repeating several above arguments one also arrives at
	\begin{align*}
		\|\mathcal{I}_j(t,\cdot)\|_{L^1(|x|\geqslant 1)} \lesssim 1.
	\end{align*}
	Eventually, we conclude the following estimate for $n=2m$:
	\begin{align*}
		\left\|\ml{M}(t,\cdot)\right\|_{L^1(|x| \geqslant 1)} \lesssim t^{-\frac{\alpha}{2}}\|\ml{H}(t,\cdot)\|_{L^1(|x|\geqslant 1)}\lesssim t^{-\frac{\alpha}{2}}.
	\end{align*}
	\end{itemize} 
	Summarizing the last estimates, Lemma \ref{Key_Lemma} is proved completely.
\end{proof}

\begin{remark}\label{Rem_n=1}
    By following some steps in our proof for odd dimensional cases without using the vector field $\ml{V}(h(r))$, we may state that the corresponding estimate for $n=1$ is still valid.
\end{remark}

    Following the proof of Lemma \ref{Key_Lemma} we may also conclude the following $L^1$ estimates, which are easier than those of Lemma \ref{Key_Lemma} since the denominator $g_{\mathrm{H}}(|\xi|)$ does not occur.
\begin{lemma} \label{Cos-Sin_Lemma}
	Let $\alpha\geqslant0$. The following estimates hold in $\R^n$ for any $n \geqslant 1$:
	\begin{align*}
		\left\|
		\ml{F}^{-1}_{\xi\to x}\left(\mathrm{e}^{-c_1 |\xi|^2 t}|\xi|^{\alpha} \sin \left(c_2 |\xi|^2 g_{\rm H}(|\xi|) t\right)\chi_{\rm H}(|\xi|)\right)(t,\cdot)\right\|_{L^1}
		&\lesssim t^{-\frac{\alpha}{2}},\\ 
		\left\|
		\ml{F}^{-1}_{\xi\to x}\left(\mathrm{e}^{-c_1 |\xi|^2 t}|\xi|^{\alpha} \cos \left(c_2 |\xi|^2 g_{\rm H}(|\xi|) t\right)\chi_{\rm H}(|\xi|)\right)(t,\cdot)\right\|_{L^1}
		&\lesssim t^{-\frac{\alpha}{2}},
	\end{align*}
	where $c_1$ is a positive constant and $c_2 \ne 0$ is a real constant.
\end{lemma}

\begin{prop} \label{L^1_High}
	Let $\alpha\geqslant0$.	The following estimates hold in $\R^n$ for any $n \geqslant 1$:
	\begin{align*}
		\left\|\ml{F}^{-1}_{\xi\to x}\left(|\xi|^\alpha \widehat{K}_j(t,|\xi|)\chi_{\rm H}(|\xi|)\right)(t,\cdot)\right\|_{L^1} &\lesssim t^{j-\frac{\alpha}{2}}
	\end{align*}
with $j=0,1$.
\end{prop}
\begin{proof}
    Thanks to the representation formulas \eqref{K_0-Formula} and \eqref{K_1-Formula}, we may link Lemma \ref{Key_Lemma} with Lemma \ref{Cos-Sin_Lemma} to achieve the desired estimates in this proposition, where we also employed the relation
	\begin{align*}
		\widehat{K}_1(t,|\xi|)= t\int_0^1 \mathrm{e}^{-\epsilon|\xi|^2 t+i(2\theta-1)t|\xi|\sqrt{(1-\epsilon^2)|\xi|^2+1}}{\rm d}\theta 
	\end{align*}
	from the Newton-Leibniz formula.
\end{proof}

\begin{prop} \label{L^ity_High}
	Let $\alpha\geqslant0$. The following estimates hold in $\R^n$ for any $n \geqslant 1$:
	\begin{align*}
		\left\|\ml{F}_{\xi\to x}^{-1}\left(|\xi|^\alpha \widehat{K}_j(t,|\xi|)\chi_{\rm H}(|\xi|)\right)(t,\cdot)\right\|_{L^\ity} &\lesssim t^{j-\frac{n+\alpha}{2}}
	\end{align*}
    with $j=0,1$.
\end{prop}
\begin{proof}
	It is clear to see that $\widehat{K}_1(t,|\xi|)$ can be written by
	\begin{align*}
	    \widehat{K}_1(t,|\xi|)= \mathrm{e}^{\lambda_+(|\xi|) t}\f{1- \mathrm{e}^{(\lambda_-(|\xi|)- \lambda_+(|\xi|))t}}{\lambda_+(|\xi|)-\lambda_-(|\xi|)}= t\,\mathrm{e}^{\lambda_+(|\xi|) t}\int_0^1 \mathrm{e}^{-2i\theta|\xi|t\sqrt{(1-\epsilon^2)|\xi|^2+1}}{\rm d}\theta.
	\end{align*}
	Then, using the asymptotic behavior of the characteristic roots we  realize
	\begin{align*}
		\left|\widehat{K}_0(t,|\xi|)\right| \lesssim \mathrm{e}^{-c|\xi|^2t}\ \ \mbox{and}\ \  \left|\widehat{K}_1(t,|\xi|)\right| \lesssim t\, \mathrm{e}^{-c|\xi|^2 t} 
	\end{align*}
	for large $|\xi|$, where $c$ is a suitable positive constant. Hence, with the aid of mapping property for the Fourier transformation, one may easily conclude that all the desired statements hold.
\end{proof}

    From Propositions \ref{L^1_High} and \ref{L^ity_High}, by employing the Riesz-Thorin interpolation theorem we immediately obtain the result. 
\begin{prop} \label{L^r_High}
    Let $\alpha\geqslant0$.	The following estimates hold in $\R^n$ for any $n \geqslant 1$:
	\begin{align*}
		\left\|\ml{F}^{-1}_{\xi\to x}\left(|\xi|^\alpha \widehat{K}_j(t,|\xi|)\chi_{\rm H}(|\xi|)\right)(t,\cdot)\right\|_{L^r} \lesssim t^{j-\frac{n}{2}(1-\frac{1}{r})- \frac{\alpha}{2}}
	\end{align*}
	with $j=0,1$ and $r\in [1,\ity]$.
\end{prop}

\subsection{Estimates for low frequencies}
\begin{prop} \label{L^1_Low}
	Let $\alpha\geqslant0$. The following estimates hold in $\R^n$ for any $n \geqslant 1$:
	\begin{align*}
		\left\|\ml{F}^{-1}_{\xi\to x}\left(|\xi|^\alpha \widehat{K}_j(t,|\xi|)\chi_{\rm L}(|\xi|)\right)(t,\cdot)\right\|_{L^1} \lesssim (1+t)^{j+\frac{1}{2}[\frac{n}{2}-j]-\frac{\alpha}{2}}
	\end{align*}
with $j=0,1$
\end{prop}
\begin{proof}
	Following the same approach as Statement 4 (Theorem 2.1) in \cite{Shibata=2000} (see also Proposition 3.7 in \cite{Dao-Reissig=2019-1} for more general setting) combined with several similar steps to the proof of Lemma \ref{Key_Lemma} we may derive the desired estimates.
\end{proof}

\begin{prop} \label{L^ity_Low}
	Let $\alpha\geqslant0$. The following estimates hold in $\R^n$ for any $n \geqslant 1$:
	\begin{align*}
		\left\|\ml{F}^{-1}_{\xi\to x}\left(|\xi|^\alpha \widehat{K}_j(t,|\xi|)\chi_{\rm L}(|\xi|)\right)(t,\cdot)\right\|_{L^\ity} \lesssim (1+t)^{j-\frac{n+\alpha}{2}}
	\end{align*}
with $j=0,1$.
\end{prop}
\begin{proof}
	As in Proposition \ref{L^ity_High}, using the asymptotic behavior of the characteristic roots we may arrive at the relations
	\begin{align*}
	\left|\widehat{K}_0(t,|\xi|)\right| \lesssim \mathrm{e}^{-c|\xi|^2t}\ \ \mbox{and}\ \  \left|\widehat{K}_1(t,|\xi|)\right| \lesssim t\, \mathrm{e}^{-c|\xi|^2 t} 
\end{align*}
	for small $|\xi|$. Then, the application of the auxiliary inequality
	\begin{align*}
		\int_{\R^n} |\xi|^\alpha \mathrm{e}^{-c|\xi|^\beta t}\chi_{\rm L}(|\xi|){\rm d}\xi\lesssim\int_0^{1}r^{\alpha+n-1}\mathrm{e}^{-cr^{\beta}t}\mathrm{d}r \lesssim (1+t)^{-\frac{n+\alpha}{\beta}}
	\end{align*}
	for $n\geqslant 1$, where $\alpha>-n$ and $\beta>0$, provides immediately the desired estimates. 
\end{proof}

    The employment of the Riesz-Thorin interpolation theorem again between Propositions \ref{L^1_Low} and \ref{L^ity_Low} leads to conclude the result.
\begin{prop} \label{L^r_Low}
	Let $\alpha\geqslant0$. The following estimates hold in $\R^n$ for any $n \geqslant 1$:
	\begin{align*}
		\left\|\ml{F}^{-1}_{\xi\to x}\left(|\xi|^\alpha \widehat{K}_j(t,|\xi|)\chi_{\rm L}(|\xi|)\right)(t,\cdot)\right\|_{L^r} &\lesssim (1+t)^{j-\frac{n}{2}(1-\frac{1}{r})+\frac{1}{2r}[\frac{n}{2}-j]-\frac{\alpha}{2}}
	\end{align*}
	with $j=0,1$ and $r\in [1,\ity]$.
\end{prop}

\subsection{Conclusion of $L^m-L^q$ estimates for solutions}
    Finally, to end this section we are going to state the $L^m-L^q$ estimates for solutions to the linearized Cauchy problem \eqref{Equation_Main_Linear}, which also provide some large-time behaviors of solutions in the $\dot{H}^s_q$ framework for any $q\geqslant 1$.
\begin{theorem} \label{Thm_L^m-L^q}
	Let $s\geqslant0$ and $1\leqslant m\leqslant q\leqslant \ity$. The solutions to \eqref{Equation_Main_Linear} fulfills the following $L^m-L^q$ estimates:
	\begin{align*}
		\|v(t,\cdot )\|_{\dot{H}^s_q}& \lesssim
		\begin{cases}
			t^{-\frac{n}{2}(1-\frac{1}{r})- \frac{s}{2}}\|v_0\|_{L^m}+ t^{1-\frac{n}{2}(1-\frac{1}{r})- \frac{s}{2}} \|v_1\|_{L^m} &\text{if} \ \  t\in(0,1], \\
			(1+t)^{-\frac{n}{2}(1-\frac{1}{r})+\frac{1}{2r}[\frac{n}{2}]-\frac{s}{2}} \|v_0\|_{L^m}+ (1+t)^{1-\frac{n}{2}(1-\frac{1}{r})+\frac{1}{2r}[\frac{n}{2}-1]-\frac{s}{2}} \|v_1\|_{L^m} &\text{if}\ \  t\in [1,\ity),
		\end{cases}\\
		\| v_t(t,\cdot )\|_{\dot{H}^s_q}& \lesssim \begin{cases}
			t^{-\frac{n}{2}(1-\frac{1}{r})-1- \frac{s}{2}}\|v_0\|_{L^m}+ t^{-\frac{n}{2}(1-\frac{1}{r})- \frac{s}{2}} \|v_1\|_{L^m} &\text{if} \ \  t\in(0,1], \\
			(1+t)^{-\frac{n}{2}(1-\frac{1}{r})+\frac{1}{2r}[\frac{n}{2}-1]-\frac{s}{2}} \|v_0\|_{L^m}+ (1+t)^{-\frac{n}{2}(1-\frac{1}{r})+\frac{1}{2r}[\frac{n}{2}]-\frac{s}{2}} \|v_1\|_{L^m} &\text{if}\ \  t\in [1,\ity),
		\end{cases}
	\end{align*}
	 for all $n\geqslant 1$, where $1+ 1/q= 1/r+1/m$.
\end{theorem}
\begin{proof}
	Thanks to the estimates in Propositions \ref{L^r_High} and \ref{L^r_Low} with the estimate \eqref{estimate_middle}, we apply Young's convolution inequality to conclude the estimate for $v(t,\cdot)$ in $\dot{H}^s_q$. Involving the estimate for the time-derivative of solutions we know that
	\begin{align*}
	\partial_t \widehat{K}_0(t,|\xi|)= -(|\xi|^2+|\xi|^4)\widehat{K}_1(t,|\xi|)\ \  \mbox{and}\ \  \partial_t \widehat{K}_1(t,|\xi|)= \widehat{K}_0(t,|\xi|)- 2\epsilon |\xi|^2\widehat{K}_1(t,|\xi|).	
	\end{align*}
	Then, we apply again Young's convolution inequality associated with Propositions \ref{L^r_High} and \ref{L^r_Low} to arrive at the second statement. This completes our proof.
\end{proof}

\begin{remark}
    Although we may not completely compare Theorem \ref{Thm_L^m-L^q} in this paper with the dissipative-dispersive estimates for the linearized viscous Boussinesq equation in  Lemma 3.1 of \cite{Liu-Wang=2020} since different assumptions on initial data,  we can estimate solutions in $\dot{H}^{s}_q$ even for $1\leqslant q<2$ or $q=\infty$. Particularly, concerning the case $1\leqslant q<2$ we realize that the embedding $\dot{B}^0_{q,2}\hookrightarrow L^q$ does not hold anymore, where $\dot{B}^s_{q,r}$ stand for the homogeneous Besov spaces. For this reason, our result is non-trivial in the sense that the new  dissipative-dispersive estimates in Theorem 2.5 of \cite{Liu-Wang=2020} do not work in our case.
\end{remark}

\section{Inviscid limits in the $L^{\infty}$ framework}\label{Section_Inviscid_Limits}
    As we mentioned in the introduction, the inviscid Boussinesq equation \eqref{Equation_Inviscid_Linear} is formally the limit equation of the viscous Boussinesq equation \eqref{Equation_Main_Linear} as the vanishing viscosity, i.e. $\epsilon=0$. This section mainly contributes to the asymptotic analysis for the small viscosity and rigorous justification of inviscid limits. For later convenience, let us denote by $v^{\epsilon}=v^{\epsilon}(t,x)$, the solution to the viscous Boussinesq equation \eqref{Equation_Main_Linear} with $0<\epsilon\ll 1$ and initial data $v^{\epsilon}(0,x)=v_0^{\epsilon}(x)$, $v_t^{\epsilon}(0,x)=v_1^{\epsilon}(x)$.
\subsection{Formal derivations of profiles with the small viscosity}
    This subsection is devoted to the application of the multi-scale analysis and WKB analysis to get the equations for the asymptotic profiles with $0<\epsilon\ll 1$ formally. Basing on the motivation of the multi-scale method (see, for example, \cite{Holmes-1995}), the solution $v^{\epsilon}(t,x)$ to \eqref{Equation_Main_Linear} with $\epsilon>0$ has the following expansions:
\begin{align}\label{expansion}
	v^{\epsilon}(t,x)=\sum\limits_{j\geqslant0}\epsilon^{j}\left(v^{I,j}(t,x)+v^{L,j}(t,\sqrt{\epsilon}x)\right),
\end{align}
    where each term in \eqref{expansion} is assumed to be smooth. Moreover, by taking $z=\sqrt{\epsilon}x$, we assume that the profiles of the remaining terms $v^{L,j}(t,z)$ are decaying (or zero) as $z\to 0$ for any $j\geqslant 0$. When $t=0$, the expansion \eqref{expansion} should satisfy the initial conditions. From the equation of viscous Boussinesq equation with the ansatz \eqref{expansion}, one notices that
    \begin{align}\label{Eq_sum}
        0&=\sum\limits_{j\geqslant 0}\epsilon^j(v^{I,j}_{tt}+v^{L,j}_{tt})-\sum\limits_{j\geqslant0} \epsilon^j(\Delta v^{I,j}+\epsilon\Delta_zv^{L,j})+\sum\limits_{j\geqslant0}\epsilon^j(\Delta^2 v^{I,j}+\epsilon^2\Delta_z^2 v^{L,j})\notag\\
    	&\quad-2\sum\limits_{j\geqslant0}\epsilon^{j+1}(\Delta v_t^{I,j}+\epsilon\Delta_zv_t^{L,j})
    \end{align}
    with the operator $\Delta_z:=\sum_{k=1}^n\partial_{z_k}^2$, where $v^{I,j}=v^{I,j}(t,x)$ and $v^{L,j}=v^{L,j}(t,z)$.
    \medskip

    \noindent \underline{Matching condition for initial data}: Let us combine initial conditions in \eqref{Equation_Main_Linear} with \eqref{expansion}, namely,
    \begin{align*}
    	v^{\epsilon}_0(x)&=v^{I,0}(0,x)+v^{L,0}(0,z)+\sum\limits_{j\geqslant1}\epsilon^{j}\left(v^{I,j}(0,x)+v^{L,j}(0,z)\right),\\
    	v^{\epsilon}_1(x)&=v_t^{I,0}(0,x)+v_t^{L,0}(0,z)+\sum\limits_{j\geqslant1}\epsilon^{j}\left(v^{I,j}_t(0,x)+v^{L,j}_t(0,z)\right).
    \end{align*}
Consequently, we may claim that $v^{I,0}(0,x)=v^{\epsilon}_0(x)$ and $v_t^{I,0}(0,x)=v^{\epsilon}_1(x)$ if $v^{L,0}(0,z)=0= v_t^{L,0}(0,z)$; $v^{I,j}(0,x)=-v^{L,j}(0,z)$ and $v^{I,j}_t(0,x)=-v^{L,j}_t(0,z)$ for $j\geqslant 1$.\medskip

\noindent \underline{First-order profile of solutions}: To get the outer solution, we collect all terms with $\ml{O}(\epsilon^0)$ in \eqref{Eq_sum} by
\begin{align*}
	v_{tt}^{I,0}-\Delta v^{I,0}+\Delta^2 v^{I,0}=-v_{tt}^{L,0}.
\end{align*}
By letting $\epsilon\downarrow 0$ (it leads to $z\to 0$ for any fixed $x$), from the property of remaining profiles as well as the matching condition of initial data, we find that $v^{I,0}(t,x)=v^{0}(t,x)$, namely, the function $v^{I,0}(t,x)$ is the solution to the inviscid Boussinesq equation \eqref{Equation_Inviscid_Linear}. At this time, one may have $v_{tt}^{L,0}(t,z)=0$ carrying $v^{L,0}(0,z)=0= v_t^{L,0}(0,z)$. So, we arrive at $v^{L,0}(t,z)=0$ for $t>0$ and $z\in\mb{R}^n$.\medskip

\noindent \underline{Second-order profile of solutions}: We now take the terms with the size $\epsilon$ in \eqref{Eq_sum} as follows:
\begin{align*}
	v_{tt}^{I,1}-\Delta v^{I,1}+\Delta^2 v^{I,1}=2\Delta v_t^{I,0}+\Delta_z v^{L,0}-v_{tt}^{L,1}.
\end{align*}
Again taking $\epsilon\downarrow0$, we derive an inhomogeneous inviscid Boussinesq equation
\begin{align}\label{Eq_Inhomogeneous_Boussinesq}
	\begin{cases}
		v_{tt}^{I,1}-\Delta v^{I,1}+\Delta^2 v^{I,1}=2\Delta v_t^{I,0},&x\in\mb{R}^n,\ t>0,\\
		v^{I,1}(0,x)=v_0^{I,1}(x),\ \ v^{I,1}_t(0,x)=v_1^{I,1}(x),&x\in\mb{R}^n,
	\end{cases}
\end{align}
and
\begin{align}\label{Eq_01}
	\begin{cases}
		v_{tt}^{L,1}=\Delta_z v^{L,0},&z\in\mb{R}^n,\ t>0,\\
		v^{L,1}(0,z)=v^{L,1}_0(z),\ \ v^{L,1}_t(0,z)=v^{L,1}_1(z),&z\in\mb{R}^n.
	\end{cases}
\end{align}
Concerning $v^{L,0}(t,z)=0$, $v_0^{L,1}(z)=-v_0^{I,1}(x)$ and $v_1^{L,1}(z)=-v_1^{I,1}(x)$, we integrate the equation in \eqref{Eq_01} twice over $[0,t]$ to get
\begin{align*}
	v^{L,1}(t,\sqrt{\epsilon}x)=-v_0^{I,1}(x)-tv_1^{I,1}(x).
\end{align*}
Actually, the solution to the inviscid Boussinesq equation \eqref{Equation_Inviscid_Linear} can be expressed by
\begin{align}\label{Solution_Inviscid_Boussinesq}
	v^0(t,x)=\underbrace{\ml{F}_{\xi\to x}^{-1}\left(\cos\left(|\xi|\sqrt{|\xi|^2+1}\,t\right)\right)}_{=:E_0(t,x)}\ast_{(x)}v_0^0(x)+\underbrace{\ml{F}_{\xi\to x}^{-1}\left(\frac{\sin\left(|\xi|\sqrt{|\xi|^2+1}\,t\right)}{|\xi|\sqrt{|\xi|^2+1}}\right)}_{=:E_1(t,x)}\ast_{(x)}v_1^0(x).
\end{align}
It leads to the solution of \eqref{Eq_Inhomogeneous_Boussinesq} having the form
\begin{align*}
	v^{I,1}(t,x)=E_0(t,x)\ast_{(x)}v_0^{I,1}(x)+E_1(t,x)\ast_{(x)}v_{1}^{I,1}(x)+2\int_0^tE_1(t-\tau,x)\ast_{(x)}\Delta v_t^0(\tau,x)\mathrm{d}\tau,
\end{align*}
where the function $v^0(t,x)$ was given in \eqref{Solution_Inviscid_Boussinesq}.\medskip

\noindent \underline{Higher-order profiles of solutions}: As we observed in the last remaining profiles in \eqref{Eq_sum}, the lowest order is $\epsilon^2$. For this reason, we collect all terms with the size $\epsilon^{j+2}$ for any $j\geqslant0$, namely,
\begin{align}\label{Eq_Higher_Order}
\begin{cases}
	v_{tt}^{I,j+2}-\Delta v^{I,j+2}+\Delta^2 v^{I,j+2}=2\Delta v_t^{I,j+1},&x\in\mb{R}^n,\ t>0,\\
	v^{I,j+2}(0,x)=v_0^{I,j+2}(x),\ \ v_t^{I,j+2}(0,x)=v_1^{I,j+2}(x),&x\in\mb{R}^n,
\end{cases}
\end{align}
where the source term $2\Delta v_t^{I,j+1}$ is determined by the $j+1$ step with the size $\epsilon^{j+1}$, moreover,
\begin{align}\label{Eq_Higher_Order_Remain}
	\begin{cases}
		v_{tt}^{L,j+2}=\Delta_z v^{L,j+1}-\Delta_z^2 v^{L,j}+2\Delta_z v_t^{L,j},&z\in\mb{R}^n,\ t>0,\\
		v^{L,j+2}(0,z)=v^{L,j+2}_0(z),\ \ v^{L,j+2}_t(0,z)=v^{L,j+2}_1(z),&z\in\mb{R}^n,\\
		v^{L,j+2}_0(z)+v_0^{I,j+2}(x)=0,\ \ v_1^{L,j+2}(z)+v_1^{I,j+2}(x)=0,&x,z\in\mb{R}^n,
	\end{cases}
\end{align}
where the terms $v^{L,j}(t,z)$ and $v^{L,j+1}(t,z)$ are given by the $j$ (with size $\epsilon^j$) and $j+1$ (with size $\epsilon^{j+1}$) steps, respectively. It results that
\begin{align}\label{Rep_Higher_profile_remain}
	v^{L,j+2}(t,z)&=v_0^{L,j+2}(z)+tv_1^{L,j+2}(z)-2t\Delta_zv_0^{L,j}(z)+2\int_0^t\Delta_zv^{L,j}(\tau,z)\mathrm{d}\tau\notag\\
	&\quad+\int_0^t\int_0^{\tau}\left(\Delta_zv^{L,j+1}(\eta,z)-\Delta_z^2v^{L,j}(\eta,z)\right)\mathrm{d}\eta\mathrm{d}\tau,
\end{align}
where $v^{L,j+2}_0(z)=-v_0^{I,j+2}(x)$ as well as $v_1^{L,j+2}(z)=-v_1^{I,j+2}(x)$.\medskip

\noindent Summarizing the last statements, we can conclude the next result for formal expansions with some profiles to the viscous Boussinesq equation \eqref{Equation_Main_Linear} with small $\epsilon>0$.
\begin{prop}\label{Prop_Formal_Expan}
	The solution $v^{\epsilon}=v^{\epsilon}(t,x)$ to the Cauchy problem for the viscous Boussinesq equation \eqref{Equation_Main_Linear} with small $\epsilon>0$ formally has the following asymptotic expansions:
	\begin{align*}
		v^{\epsilon}(t,x)=v^0(t,x)+\epsilon \left(v^{I,1}(t,x)-v_0^{I,1}(x)-tv_1^{I,1}(x)\right)+\sum\limits_{j\geqslant0}\epsilon^{j+2}\left(v^{I,j+2}(t,x)+v^{L,j+2}(t,\sqrt{\epsilon}x)\right),
	\end{align*}
where $v^0=v^0(t,x)$ is the solution to the inviscid Boussinesq equation \eqref{Equation_Inviscid_Linear}; $v^{I,1}=v^{I,1}(t,x)$ is the solution to the inhomogeneous inviscid Boussinesq equation \eqref{Eq_Inhomogeneous_Boussinesq} with initial value $v_0^{I,1}(x)$, $v_1^{I,1}(x)$; $v^{I,j+2}=v^{I,j+2}(t,x)$ is the solution to the Cauchy problem \eqref{Eq_Higher_Order}; and the function $v^{L,j+2}(t,\sqrt{\epsilon}x)=v^{L,j+2}(t,z)$ can be represented by \eqref{Rep_Higher_profile_remain}.
\end{prop}

Since the  complexity of Proposition \ref{Prop_Formal_Expan}, it seems challenging to demonstrate the expansion for all $j\geqslant 0$. Hence, we will rigorously justify the first-order profiles (classical inviscid limit argument) and the second-order profiles in the next two subsections.

\subsection{Rigorous justification of inviscid limits for the first-order profile}
\begin{theorem}\label{Thm_Convergence}
	Let $0<\epsilon\ll 1$. Let us assume $v^{\epsilon}_0(x)\equiv v^0_0(x)$ and $v^{\epsilon}_1(x)\equiv v^0_1(x)$ fulfilling $\langle D\rangle^{s_0}v_0^0\in L^1$ and $\langle D\rangle^{s_1}v_1^0\in L^1$ for $ s_0>n+2$, $s_1>n$ with $n\geqslant 1$. Then, the solution $v^{\epsilon}=v^{\epsilon}(t,x)$ to the viscous Boussinesq equation \eqref{Equation_Main_Linear} converges to the one $v^0(t,x)$ to the inviscid Boussinesq equation \eqref{Equation_Inviscid_Linear} in the following sense:
	\begin{align*}
		\sup\limits_{t\in[0,T],\ x\in\R^n}|v^{\epsilon}(t,x)-v^0(t,x)|\leqslant C_{T}\,\epsilon\left(\|\langle D\rangle^{s_0}v_0^0\|_{L^1}+\|\langle D\rangle^{s_1}v_1^0\|_{L^1}\right),
	\end{align*}
	where the constant $C_T$ is independent of $\epsilon$.
\end{theorem}
\begin{remark} In Theorem \ref{Thm_Convergence}, the convergence in any finite time such that
	\begin{align*}
		v^{\epsilon}\to v^0 \ \ \mbox{in}\ \ L^{\infty}([0,T]\times\mb{R}^n)\ \ \mbox{as}\ \ \epsilon\downarrow0
	\end{align*}
has been shown. Namely, the inviscid limit holds with the rate of convergence $\epsilon$. Furthermore, it also implies the correction of Proposition \ref{Prop_Formal_Expan} for the first-order since $\sup|v^{\epsilon}(t,x)-v^0(t,x)|=\ml{O}(\epsilon)$.
\end{remark}
\begin{proof}
	According to the Fourier images of the solutions to \eqref{Equation_Main_Linear} and \eqref{Equation_Inviscid_Linear}, let us take directly the difference between them
	\begin{align*}
			|\hat{v}^{\epsilon}(t,\xi)-\hat{v}^0(t,\xi)|\leqslant C\sum\limits_{\pm}\left(I_0^{\pm}(t,|\xi|)|\hat{v}_0^0(\xi)|+I_1^{\pm}(t,|\xi|)|\hat{v}_1^0(\xi)|\right),
	\end{align*}
where we defined
\begin{align*}
	I_0^{\pm}(t,|\xi|)&:=\left|\frac{\left(-\epsilon|\xi|^2\pm i|\xi|\sqrt{(1-\epsilon^2)|\xi|^2+1}\,
		\right)\mathrm{e}^{-\epsilon|\xi|^2t\mp i|\xi|\sqrt{(1-\epsilon^2)|\xi|^2+1}t}}{2i|\xi|\sqrt{(1-\epsilon^2)|\xi|^2+1}}-\frac{\left(\pm i|\xi|\sqrt{|\xi|^2+1}\,
		\right)\mathrm{e}^{\mp i|\xi|\sqrt{|\xi|^2+1}t}}{2i|\xi|\sqrt{|\xi|^2+1}}\right|,\\
	I_1^{\pm}(t,|\xi|)&:=\left|\frac{\mathrm{e}^{-\epsilon|\xi|^2t\pm i|\xi|\sqrt{(1-\epsilon^2)|\xi|^2+1}t}}{2i|\xi|\sqrt{(1-\epsilon^2)|\xi|^2+1}}-\frac{\mathrm{e}^{\pm i|\xi|\sqrt{|\xi|^2+1}t}}{2i|\xi|\sqrt{|\xi|^2+1}}\right|.
\end{align*}
	One may notice a crucial fact
	\begin{align}\label{Fact}
		\sqrt{(1-\epsilon^2)|\xi|^2+1}-\sqrt{|\xi|^2+1}=\frac{-\epsilon^2|\xi|^2}{\sqrt{(1-\epsilon^2)|\xi|^2+1}+\sqrt{|\xi|^2+1}}\sim\frac{-\epsilon^2|\xi|^2}{2\sqrt{|\xi|^2+1}}.
	\end{align}
	For one thing, constructing an auxiliary term implies
	\begin{align*}
		I_1^{\pm}(t,|\xi|)&\leqslant \left|\frac{\mathrm{e}^{-\epsilon|\xi|^2t\pm i|\xi|\sqrt{(1-\epsilon^2)|\xi|^2+1}t}}{2i|\xi|\sqrt{(1-\epsilon^2)|\xi|^2+1}}-\frac{\mathrm{e}^{\pm i|\xi|\sqrt{|\xi|^2+1}t}}{2i|\xi|\sqrt{(1-\epsilon^2)|\xi|^2+1}}\right|\\
		&\quad+\left|\frac{\mathrm{e}^{\pm i|\xi|\sqrt{|\xi|^2+1}t}}{2i|\xi|\sqrt{(1-\epsilon^2)|\xi|^2+1}}-\frac{\mathrm{e}^{\pm i|\xi|\sqrt{|\xi|^2+1}t}}{2i|\xi|\sqrt{|\xi|^2+1}}\right|\\
		&\leqslant\left|\frac{\mathrm{e}^{\pm i|\xi|\sqrt{|\xi|^2+1}t}}{2i|\xi|\sqrt{(1-\epsilon^2)|\xi|^2+1}}\left(\mathrm{e}^{-\epsilon|\xi|^2t\pm i|\xi|\left(\sqrt{(1-\epsilon^2)|\xi|^2+1}-\sqrt{|\xi|^2+1}\right)t}-1\right)\right|\\
		&\quad+\left|\frac{\mathrm{e}^{\pm i|\xi|\sqrt{|\xi|^2+1}t}}{2i|\xi|} \cdot \frac{\sqrt{|\xi|^2+1}-\sqrt{(1-\epsilon^2)|\xi|^2+1}}{\sqrt{(1-\epsilon^2)|\xi|^2+1}\sqrt{|\xi|^2+1}}\right|.
	\end{align*}
	According to the equality
	\begin{align*}
		\mathrm{e}^{-\epsilon|\xi|^2t\pm i|\xi|\left(\sqrt{(1-\epsilon^2)|\xi|^2+1}-\sqrt{|\xi|^2+1}\right)t}-1&=\left(-\epsilon|\xi|^2t\pm i|\xi|\left(\sqrt{(1-\epsilon^2)|\xi|^2+1}-\sqrt{|\xi|^2+1}\right)t\right)\\
		&\ \quad\times\int_0^1\mathrm{e}^{-\epsilon|\xi|^2t\eta\pm i|\xi|\left(\sqrt{(1-\epsilon^2)|\xi|^2+1}-\sqrt{|\xi|^2+1}\right)t\eta}\mathrm{d}\eta,
	\end{align*}
	and the fact \eqref{Fact}, we arrive at
	\begin{align*}
		I_1^{\pm}(t,|\xi|)&\leqslant C\left(\epsilon|\xi|\langle\xi\rangle^{-1}t+\epsilon^2|\xi|^2\langle\xi\rangle^{-2}t+\epsilon^2|\xi|\langle\xi\rangle^{-3}\right).
	\end{align*}
	By using the same approach as the above, we introduce two auxiliary terms to treat the difference $I_0^{\pm}(t,|\xi|)$. Precisely, we know
	\begin{align*}
		&\left|\frac{\left(-\epsilon|\xi|^2\pm i|\xi|\sqrt{(1-\epsilon^2)|\xi|^2+1}\,
			\right)\mathrm{e}^{-\epsilon|\xi|^2t\mp i|\xi|\sqrt{(1-\epsilon^2)|\xi|^2+1}t}}{2i|\xi|\sqrt{(1-\epsilon^2)|\xi|^2+1}}-\frac{\pm \sqrt{|\xi|^2+1}
			\mathrm{e}^{-\epsilon|\xi|^2t\mp i|\xi|\sqrt{(1-\epsilon^2)|\xi|^2+1}t}}{2\sqrt{(1-\epsilon^2)|\xi|^2+1}}\right|\\
		&\qquad\leqslant\frac{C\mathrm{e}^{-\epsilon|\xi|^2t}}{|\xi|\sqrt{(1-\epsilon^2)|\xi|^2+1}}\left(\epsilon|\xi|^2+|\xi| \left|\sqrt{(1-\epsilon^2)|\xi|^2+1}-\sqrt{|\xi|^2+1}\,
		\right|\right)\\
		&\qquad\leqslant C\epsilon|\xi|\langle\xi\rangle^{-1}\left(1+\epsilon|\xi|\langle\xi\rangle^{-1}\right),
	\end{align*}
	and
	\begin{align*}
	&\frac{1}{2}\left|\frac{\pm \sqrt{|\xi|^2+1}
		\mathrm{e}^{-\epsilon|\xi|^2t\mp i|\xi|\sqrt{(1-\epsilon^2)|\xi|^2+1}t}}{\sqrt{(1-\epsilon^2)|\xi|^2+1}}\mp\mathrm{e}^{-\epsilon|\xi|^2t\mp i|\xi|\sqrt{(1-\epsilon^2)|\xi|^2+1}t}\right|\\
	&\qquad+\frac{1}{2}\left|\pm
		\mathrm{e}^{-\epsilon|\xi|^2t\mp i|\xi|\sqrt{(1-\epsilon^2)|\xi|^2+1}t}\mp\mathrm{e}^{\mp i|\xi|\sqrt{|\xi|^2+1}t}\right|\leqslant C\epsilon^2|\xi|^2\langle\xi\rangle^{-2}+C\epsilon|\xi|^2t\left(1+|\xi|\langle\xi\rangle^{-1}\right).
	\end{align*}
	One can derive
	\begin{align*}
		I_0^{\pm}(t,|\xi|)\leqslant C\left(\epsilon|\xi|\langle\xi\rangle^{-1}+\epsilon^2|\xi|^2\langle\xi\rangle^{-2}+\epsilon|\xi|^2t+\epsilon^2|\xi|^3\langle\xi\rangle^{-1}t\right).
	\end{align*}
	Thus, by using the Hausdorff-Young inequality, we separate the phase into two parts as follows:
	\begin{align*}
		&\|v^{\epsilon}(t,\cdot)-v^0(t,\cdot)\|_{L^{\infty}}\leqslant C\|\hat{v}^{\epsilon}(t,\xi)-\hat{v}^0(t,\xi)\|_{L^1}\\
		&\quad\leqslant C\sum\limits_{\pm}\left(\left\|\langle\xi\rangle^{-s_0}I_0^{\pm}(t,|\xi|)\right\|_{L^1}\|\langle D\rangle^{s_0}v_0\|_{L^1}+ \left\|\langle\xi\rangle^{-s_1}I_1^{\pm}(t,|\xi|)\right\|_{L^1}\|\langle D\rangle^{s_1}v_1\|_{L^1}\right)\\
		&\quad\leqslant C\epsilon\left(\int_{|\xi|\leqslant 1}|\xi|\left(1+\epsilon|\xi|+|\xi|t+\epsilon|\xi|^2t\right)\mathrm{d}\xi+\int_{|\xi|\geqslant 1}\left(1+\epsilon+\langle\xi\rangle^2t+\epsilon\langle\xi\rangle^2t\right)\langle\xi\rangle^{-s_0}\mathrm{d}\xi\right)\|\langle D\rangle^{s_0}v_0^0\|_{L^1}\\
		&\quad\quad+C\epsilon\left(\int_{|\xi|\leqslant 1}|\xi|\left(t+\epsilon|\xi|t+\epsilon\right)\mathrm{d}\xi+\int_{|\xi|\geqslant 1}\left(t+\epsilon t+\epsilon\langle\xi\rangle^{-2}\right)\langle\xi\rangle^{-s_1}\mathrm{d}\xi\right)\|\langle D\rangle^{s_1}v_1^0\|_{L^1}\\
		&\quad\leqslant C\epsilon(1+t)\left(\|\langle D\rangle^{s_0}v_0^0\|_{L^1}+\|\langle D\rangle^{s_1}v_1^0\|_{L^1}\right)
	\end{align*}
	for $n\geqslant 1$ with $s_0>n+2$ and $s_1>n$. In conclusion, our proof is complete.
\end{proof} 
\subsection{Rigorous justification of inviscid limits for the second-order profile}\label{Sub_Sec_Second-Order}
\begin{theorem}\label{Thm_Second_Oreder_Expan}
		Let $0<\epsilon\ll 1$. Let us assume $v^{\epsilon}_0(x)\equiv v^0_0(x)$ and $v^{\epsilon}_1(x)\equiv v^0_1(x)$ fulfilling $\langle D\rangle^{s_0}v_0^0\in L^1$ and $\langle D\rangle^{s_1}v_1^0\in L^1$ for $ s_0>n+4$, $s_1>n+2$ with $n\geqslant 1$. Then, the solution $v^{\epsilon}=v^{\epsilon}(t,x)$ to the viscous Boussinesq equation \eqref{Equation_Main_Linear} converges to the one $v^0(t,x)$ to the inviscid Boussinesq equation \eqref{Equation_Inviscid_Linear} with a higher-order profile in the following way:
	\begin{align*}
		\sup\limits_{t\in[0,T],\ x\in\R^n}\left|v^{\epsilon}(t,x)-v^0(t,x)-\epsilon\tilde{v}(t,x)\right|\leqslant C_{T}\,\epsilon^2\left(\|\langle D\rangle^{s_0}v_0^0\|_{L^1}+\|\langle D\rangle^{s_1}v_1^0\|_{L^1}\right),
	\end{align*}
	where the constant $C_T$ is independent of $\epsilon$, and
	\begin{align*}
\tilde{v}(t,x):=2\int_0^tE_1(t-\tau,x)\ast_{(x)}\Delta v_t^0(\tau,x)\mathrm{d}\tau.
	\end{align*} Here, $E_1=E_1(t,x)$ is the fundamental solution to the inviscid Boussinesq equation \eqref{Equation_Inviscid_Linear} with initial data $v_0^0(x)=0$ as well as $v_1^0(x)=\delta_0$.
\end{theorem}
\begin{remark} In Theorem \ref{Thm_Second_Oreder_Expan}, the convergence in any finite time such that
	\begin{align*}
		v^{\epsilon}\to v^0+\epsilon\tilde{v} \ \ \mbox{in}\ \ L^{\infty}([0,T]\times\mb{R}^n)\ \ \mbox{as}\ \ \epsilon\downarrow0
	\end{align*}
	has been shown. Namely, the inviscid limit holds with the rate of convergence $\epsilon^2$. Furthermore, it also implies the correction of Proposition \ref{Prop_Formal_Expan} for the second-order if $v_j^{I,1}(x)\equiv0$ for $j=0,1$.
\end{remark}
\begin{remark}
	Comparing with the inviscid limit result in Theorem \ref{Thm_Convergence}, we found that in Theorem \ref{Thm_Second_Oreder_Expan}, by subtracting the additional error term $\epsilon\tilde{v}$, which is a solution to inhomogeneous inviscid Boussinesq equation carrying the source term $2\Delta v_t^0$, the rate of convergence has been improved by a factor $\epsilon$ in $L^{\infty}([0,T]\times\mb{R}^n)$.
\end{remark}
\begin{proof}
To begin with, let us introduce a new ansatz $\widehat{R}^{\epsilon}=\widehat{R}^{\epsilon}(t,\xi)$ such that
\begin{align*}
	\widehat{R}^{\epsilon}(t,\xi)=\hat{v}^{\epsilon}(t,\xi)-\hat{v}^0(t,\xi)-\epsilon\hat{v}^{I,1}(t,\xi),
\end{align*}
where $\hat{v}^{I,1}=\hat{v}^{I,1}(t,\xi)$ is the Fourier image of solution to the inhomogeneous Cauchy problem \eqref{Eq_Inhomogeneous_Boussinesq} with vanishing data, namely,
\begin{align*}
	\begin{cases}
		\hat{v}^{I,1}_{tt}+(|\xi|^2+|\xi|^4)\hat{v}^{I,1}=-2|\xi|^2\hat{v}^{0}_t,&\xi\in\mb{R}^n,\ t>0,\\
		\hat{v}^{I,1}(0,\xi)=0,\ \ \hat{v}^{I,1}_t(0,\xi)=0,&\xi\in\mb{R}^n.
	\end{cases}
\end{align*}
Therefore, the remainder $\widehat{R}^{\epsilon}$ satisfies
\begin{align*}
	\widehat{R}^{\epsilon}_{tt}+(|\xi|^2+|\xi|^4)\widehat{R}^{\epsilon}=-2\epsilon|\xi|^2(\hat{v}_t^{\epsilon}-\hat{v}_t^0).
\end{align*}
Multiplying the above equation by $2\overline{\widehat{R}^{\epsilon}_t}$, where $\overline{\widehat{R}^{\epsilon}}$ is the conjugate of $\widehat{R}^{\epsilon}$, and taking the real part of the resultant, we immediately derive
\begin{align}\label{Ineq_01}
	\frac{\mathrm{d}}{\mathrm{d}t}\left(|\widehat{R}^{\epsilon}_t|^2+(|\xi|^2+|\xi|^4)|\widehat{R}^{\epsilon}|^2\right)&=-4\epsilon|\xi|^2\Re\left((\hat{v}_t^{\epsilon}-\hat{v}_t^0)\overline{\widehat{R}^{\epsilon}_t}\,\right)\notag\\
	&\leqslant |\widehat{R}^{\epsilon}_t|^2+4\epsilon^2|\xi|^4|\hat{v}_t^{\epsilon}-\hat{v}_t^0|^2.
\end{align}
With the aim at estimating $|\hat{v}_t^{\epsilon}-\hat{v}_t^0|$, according to the representation of solutions and compatibility of initial data, we just need to control the time-derivative of kernels
\begin{align*}
	\ml{E}_j(t,|\xi|):=|\partial_t\widehat{K}_j(t,|\xi|)-\partial_t\widehat{E}_j(t,|\xi|)|\ \ \mbox{for} \ \ j=0,1.
\end{align*}
Concerning the second kernel, i.e. $j=1$ in the last line, we notice that
\begin{align*}
	\ml{E}_1(t,|\xi|)&\leqslant\epsilon|\xi|^2|\widehat{K}_1(t,|\xi|)|+\left|\mathrm{e}^{-\epsilon|\xi|^2 t}\cos\left(|\xi|\sqrt{(1-\epsilon^2)|\xi|^2+1}\,t\right)-\cos\left(|\xi|\sqrt{|\xi|^2+1}\,t\right)\right|\\
	&\leqslant C \epsilon|\xi|\langle\xi\rangle^{-1}+\mathrm{e}^{-\epsilon|\xi|^2t}\left|\cos\left(|\xi|\sqrt{(1-\epsilon^2)|\xi|^2+1}\,t\right)-\cos\left(|\xi|\sqrt{|\xi|^2+1}\,t\right) \right|\\
	&\quad+\left| \cos\left(|\xi|\sqrt{|\xi|^2+1}\,t\right)\right|\left(1-\mathrm{e}^{-\epsilon|\xi|^2t}\right).
\end{align*}
Due to the facts that
\begin{align}\label{Fact_2}
	&\left|\cos\left(|\xi|\sqrt{(1-\epsilon^2)|\xi|^2+1}\,t\right)-\cos\left(|\xi|\sqrt{|\xi|^2+1}\,t\right) \right|\notag\\
	&\qquad\leqslant C\left|\sin\left(\frac{\epsilon^2|\xi|^3t}{2\sqrt{(1-\epsilon^2)|\xi|^2+1}+2\sqrt{|\xi|^2+1}}\right)  \right|\leqslant C\epsilon^2|\xi|^3\langle\xi\rangle^{-1}t,
\end{align}
and
\begin{align}\label{Fact_1}
	1-\mathrm{e}^{-\epsilon|\xi|^2t}=\epsilon|\xi|^2t\int_0^1\mathrm{e}^{-\epsilon|\xi|^2t\tau}\mathrm{d}\tau\leqslant C\epsilon |\xi|^2t,
\end{align}
we may conclude
\begin{align*}
	\ml{E}_1(t,|\xi|)\leqslant C\epsilon|\xi|\left(\langle\xi\rangle^{-1}+\epsilon|\xi|^2\langle\xi\rangle^{-1}t+|\xi|t\right) \leqslant C\epsilon|\xi|\left(\langle\xi\rangle^{-1}+|\xi|t\right).
\end{align*}
Concerning another kernel we may estimate
\begin{align*}
	\ml{E}_0(t,|\xi|)&\leqslant \epsilon|\xi|^2|\widehat{K}_0(t,|\xi|)|+\epsilon|\xi|^2\mathrm{e}^{-\epsilon|\xi|^2t}\left|\cos\left(|\xi|\sqrt{(1-\epsilon^2)|\xi|^2+1}\,t\right)\right|\\
	&\quad+|\xi|\left|\mathrm{e}^{-\epsilon|\xi|^2t}\sqrt{(1-\epsilon^2)|\xi|^2+1}\sin\left(|\xi|\sqrt{(1-\epsilon^2)|\xi|^2+1}\,t\right)-\sqrt{|\xi|^2+1}\sin\left(|\xi|\sqrt{|\xi|^2+1}\,t\right)\right|\\
	&\leqslant C\epsilon|\xi|^2\left(1+\epsilon|\xi|\langle\xi\rangle^{-1}\right)+ |\xi|\left|\sqrt{(1-\epsilon^2)|\xi|^2+1}\sin\left(|\xi|\sqrt{(1-\epsilon^2)|\xi|^2+1}\,t\right)\left(\mathrm{e}^{-\epsilon|\xi|^2t}-1\right)\right|\\
	&\quad +|\xi|\left|\sin\left(|\xi|\sqrt{(1-\epsilon^2)|\xi|^2+1}\,t\right)\left(\sqrt{(1-\epsilon^2)|\xi|^2+1}-\sqrt{|\xi|^2+1}\right)\right|\\
	&\quad +|\xi|\langle\xi\rangle\left|\sin\left(|\xi|\sqrt{(1-\epsilon^2)|\xi|^2+1}\,t\right)-\sin\left(|\xi|\sqrt{|\xi|^2+1}\,t\right) \right|\\
	&\leqslant C\epsilon|\xi|^2+ \color{black} C|\xi|\langle\xi\rangle\left(1-\mathrm{e}^{-\epsilon|\xi|^2t}\right)+ |\xi|\left|\sqrt{(1-\epsilon^2)|\xi|^2+1}-\sqrt{|\xi|^2+1}\right|\\
	&\quad+ |\xi|\langle\xi\rangle\left|\sin\left(|\xi|\sqrt{(1-\epsilon^2)|\xi|^2+1}\,t\right)-\sin\left(|\xi|\sqrt{|\xi|^2+1}\,t\right) \right|.
\end{align*}
By using \eqref{Fact}, \eqref{Fact_1} and
\begin{align*}
	\left|\sin\left(|\xi|\sqrt{(1-\epsilon^2)|\xi|^2+1}\,t\right)-\sin\left(|\xi|\sqrt{|\xi|^2+1}\,t\right) \right|\leqslant C\left|\sin\left(\frac{\epsilon^2|\xi|^3t}{2\sqrt{(1-\epsilon^2)|\xi|^2+1}+2\sqrt{|\xi|^2+1}}\right)  \right| 
\end{align*}
associated with \eqref{Fact_2}, we arrive at
\begin{align*}
	\ml{E}_0(t,|\xi|)&\leqslant C\epsilon|\xi|^2\left(1+|\xi|\langle\xi\rangle t+\epsilon|\xi|\langle\xi\rangle^{-1}+\epsilon|\xi|^2t\right) \leqslant C\epsilon|\xi|^2\left(1+|\xi|\langle\xi\rangle t\right).
\end{align*}
In conclusion, it holds
\begin{align*}
	|\hat{v}^{\epsilon}_t(t,\xi)-\hat{v}^{0}_t(t,\xi)|^2&\leqslant C\ml{E}_0(t,|\xi|)^2|\hat{v}_0^0(\xi)|^2+C\ml{E}_1(t,|\xi|)^2|\hat{v}_1^0(\xi)|^2\\
	&\leqslant C\epsilon^2|\xi|^4\left(1+ |\xi|^2\langle\xi\rangle^{2}t^2\right)|\hat{v}_0^0(\xi)|^2+ C\epsilon^2|\xi|^2\left(\langle\xi\rangle^{-2}+|\xi|^2t^2\right)|\hat{v}_1^0(\xi)|^2\\
	&\leqslant C\epsilon^2(1+t^2)|\xi|^4\langle\xi\rangle^4|\hat{v}_0^0(\xi)|^2+C\epsilon^2(1+t^2)|\xi|^2\langle\xi\rangle^2|\hat{v}_1^0(\xi)|^2.
\end{align*}
Let us back to the frame \eqref{Ineq_01} so that
\begin{align*}
	&\frac{\mathrm{d}}{\mathrm{d}t}\left(|\widehat{R}^{\epsilon}_t|^2+(|\xi|^2+|\xi|^4)|\widehat{R}^{\epsilon}|^2\right)-\left(|\widehat{R}^{\epsilon}_t|^2+(|\xi|^2+|\xi|^4)|\widehat{R}^{\epsilon}|^2\right)\\
    &\qquad\leqslant C\epsilon^4(1+t^2)|\xi|^8\langle\xi\rangle^{4}|\hat{v}_0^0(\xi)|^2+C\epsilon^4(1+t^2)|\xi|^6\langle\xi\rangle^2|\hat{v}_1^0(\xi)|^2.
\end{align*}
Then, the employment of Gr\"onwall's inequality in the last line gives
\begin{align*}
	|\widehat{R}^{\epsilon}_t|^2+(|\xi|^2+|\xi|^4)|\widehat{R}^{\epsilon}|^2&\leqslant C\epsilon^4 \left(|\xi|^8\langle\xi\rangle^{4}|\hat{v}_0^0(\xi)|^2+ |\xi|^6\langle\xi\rangle^2|\hat{v}_1^0(\xi)|^2\right)\int_0^t\mathrm{e}^{t-\tau}(1+\tau^2)\mathrm{d}\tau\\
	&\leqslant C^*_T\,\epsilon^4(|\xi|^2+|\xi|^4)\left(|\xi|^6\langle\xi\rangle^{2}|\hat{v}_0^0(\xi)|^2+ |\xi|^4|\hat{v}_1^0(\xi)|^2\right),
\end{align*}
with $C_T^*>0$ independent of $\epsilon$, in which we considered the vanishing initial data for $\widehat{R}^{\epsilon}$. Namely,
\begin{align*}
|\widehat{R}^{\epsilon}(t,\xi)|\leqslant \sqrt{C^*_T}\,\epsilon^2\left(\langle\xi\rangle^{4}|\hat{v}_0^0(\xi)|+\langle\xi\rangle^2|\hat{v}_1^0(\xi)|\right).
\end{align*}
By using the Hausdorff-Young inequality similarly as those in the proof of Theorem \ref{Thm_Convergence}, we can derive
\begin{align*}
	\|R^{\epsilon}(t,\cdot)\|_{L^{\infty}}\leqslant C\|\widehat{R}^{\epsilon}(t,\xi)\|_{L^1}\leqslant C_T\,\epsilon^2\left(\|\langle D\rangle^{s_0}v_0^0\|_{L^1}+\|\langle D\rangle^{s_1}v_1^0\|_{L^1}\right)
\end{align*}
for $n\geqslant 1$ with $s_0> n+4$ and $s_1>n+2$. Finally, recalling
\begin{align*}
	R^{\epsilon}(t,x)&=v^{\epsilon}(t,x)-v^0(t,x)-\epsilon v^{I,1}(t,x)\\
	&=v^{\epsilon}(t,x)-v^0(t,x)-2\epsilon\int_0^tE_1(t-\tau,x)\ast_{(x)}\Delta v_t^0(\tau,x)\mathrm{d}\tau
\end{align*}
we have just completed the proof.
\end{proof}

\section{Global (in time) well-posedness for the nonlinear problem}\label{Section_GESDS}
Our main result for global (in time) existence of small data Sobolev solutions reads as follows.

\begin{theorem} \label{Global_Theorem}
	Let $q \in (1,\ity)$ be a fixed constant, $m\in [1,q)$ and $s >n/q$. We assume that $p\geqslant q/m$ for all $n> 2m(1+\kappa)$, where $\kappa:= \frac{1}{2}(\frac{1}{m}-\frac{1}{q})[\frac{n}{2}-1]$. Moreover, the exponent $p$ fulfills the condition
	\begin{align}\label{Exponent_Condition}
		p> 1+\max\left\{\f{n(1-\frac{m}{q})+ms}{n-2m(1+\kappa)},s\right\}.
	\end{align}
	Then, there exists a sufficiently small constant $\e_0>0$ such that for any initial data
	$$ (u_0,u_1) \in \mathcal{D}:= \left(L^m \cap H^{s}_q\right) \times \left(L^m \cap H^{\max\{s-2,0\}}_q\right) $$
	satisfying $\|(u_0,u_1)\|_{\mathcal{D}}\leqslant \e_0,$ 
	the nonlinear viscous Boussinesq equation \eqref{Equation_Main} admits a unique global (in time) Sobolev solution in the class
	$$ u \in \ml{C}\left([0,\ity),H^s_q\right).$$ 
	Furthermore, the following estimates hold:
	\begin{align*}
		\|u(t,\cdot)\|_{L^q}&\lesssim (1+t)^{1-\frac{n}{2}(\frac{1}{m}-\frac{1}{q})+\kappa} \|(u_0,u_1)\|_{\mathcal{D}},\\
		\| u(t,\cdot)\|_{\dot{H}^s_q}&\lesssim (1+t)^{1-\frac{n}{2}(\frac{1}{m}-\frac{1}{q})+\kappa-\frac{s}{2}} \|(u_0,u_1)\|_{\mathcal{D}}.
	\end{align*}
\end{theorem}

\begin{remark}
Let us explain the conditions for $p$ appearing in Theorem \ref{Global_Theorem}. The presence of $p\geqslant q/m$ is due to the application of the fractional Gagliardo-Nirenberg inequality (cf. Proposition \ref{Fractional_GN}). In addition, by applying the fractional powers rule (cf. Proposition \ref{Fractional_Powers}) associated with the fractional Sobolev embedding (cf. Proposition \ref{Fractional_Embedding}), we require the condition $p> 1+ s$ as long as the assumption $s >n/q$ holds. Finally, the remaining condition for $p$ in \eqref{Exponent_Condition} is to entail the same decay estimates for solutions to \eqref{Equation_Main} as those for solutions to the corresponding linear model with vanishing right-hand side (see later, Proposition \ref{Linear_Estimates_Prop}). This means no loss of decay, so we can understand the nonlinear term as a small perturbation. At this point, we want to underline that one can relax this restriction for $p$ by using the strategy which allows some loss of decay (see more \cite{Dao-Reissig=2019-1}), nevertheless, we have to pay another condition for the dimension $n$. Hence, this purpose is out of our interest.
\end{remark}

\begin{remark}
It is obvious to recognize that Theorem \ref{Global_Theorem} gives the global (in time) well-posedness in the $L^q$ framework for any $q\in(1,\infty)$ and some $L^q$ estimates for solutions to the nonlinear viscous Boussinesq equation \eqref{Equation_Main} with small data, which is really different from the view of previous results devoting to the $L^2$ setting only. With the choice of $n=3$ and $m=1$ we can observe the following examples:
\begin{itemize}
\item If we take $q= 1.01$, then we may choose $s=3$. From Theorem \ref{Global_Theorem} we obtain $p>4.03$.
\item If we take $s= \varepsilon$ for any small value of $\varepsilon>0$, then we may choose $q= 4/\varepsilon$. From Theorem \ref{Global_Theorem} we obtain $p\geqslant 4/\varepsilon$.
\end{itemize}
The point worth noticing in the second example is that thanks to the flexibility of the choice of parameter $q \in (1,\ity)$, we have derived a result for the global (in time) Sobolev solution belonging to the class
$$ u \in \ml{C}\left([0,\ity),H^s_q\right)$$
with an arbitrarily small regularity. More transparently, when two parameters $n$ and $m$ are chosen appropriately, let us illustrate the global existence range in Theorem \ref{Global_Theorem} depicting the interaction between either the admissible exponent $p$ and the regularity $s$ (for a fixed $q$ value) or the admissible exponent $p$ and the parameter $q$ (for a fixed $s$ value) in the following figure:

%============================================================================Fig.1
\begin{figure}[H]
\begin{center}
\begin{tikzpicture}[>=latex,xscale=0.95,scale=1.1]
%================================================Fig a
\draw[->] (0,0) -- (5,0)node[below]{$s$};
\draw[->] (0,0) -- (0,5)node[left]{$p$};
\node[below left] at(0,0){$0$};
\node[below] at (2.5,-.5) {{\footnotesize a. When $n$, $m$ and $q$ are fixed}};

\fill[color=black!15!white] (0.5,4.8) -- (0.5,2.7) -- (2.5,3.5) -- (3.8,4.8)--cycle;

\fill[color=black!15!white] (2.7,1)--(2.7,1.3)--(3,1.3)--(3,1)--cycle;
\node[right] at (3,1.15) {{\footnotesize \text{Global existence}}};

\draw[<-] (3.3,4.15) -- (3.7,4.15)node[right]{$p=s+1$};

\draw[<-] (1.5,3) -- (2.7,3)node[right]{$\Gamma_1$};

%=====================================s
\draw[fill] (0.5,0) circle[radius=1pt];
\node[below] at (0.5,0){{\scriptsize $\f{n}{q}$}};
\draw[dashed] (0.5,0)--(0.5,2.7);
\draw[thin] (0.5,2.7)--(0.5,4.8);

%=====================================p
\draw[fill] (0,1) circle[radius=1pt];
\node[left] at (0,1){{\scriptsize $1$}};
\draw[dashed] (0,1)--(2.5,3.5);
\draw[thin] (2.5,3.5)--(3.8,4.8);

\draw[fill] (0,2) circle[radius=1pt];
\node[left] at (0,2){{\scriptsize $\frac{q}{m}$}};
\draw[dashed] (0,2)--(3.5,2);

\draw[dashed] (0,2.5)--(0.5,2.7);
\draw[thin] (0.5,2.7)--(2.5,3.5);

%===================================================Fig b
\draw[->] (7,0) -- (12,0)node[below]{$q$};
\draw[->] (7,0) -- (7,5)node[left]{$p$};
\node[below left] at(7,0){$0$};
\node[below] at (9.5,-.5) {{\footnotesize b. When $n$, $m$ and $s$ are fixed}};

\fill[domain=2.54:3.86,variable=\p,color=black!15!white] plot ({4.65+13.7/\p},\p)-- (8.2,3.86) --(8.2,4.8)--(11.8,4) -- (10.05,2.54)--cycle;

\fill[color=black!15!white] (9.7,1)--(9.7,1.3)--(10,1.3)--(10,1)--cycle;
\node[right] at (10,1.15) {{\footnotesize \text{Global existence}}};

\draw[<-] (11.3,3.5) -- (11.6,3.5)node[right]{$p=\f{q}{m}$};

\draw[<-] (9.5,3) -- (11,3)node[right]{$\Gamma_2$};

%=====================================q
\draw[fill] (8.2,0) circle[radius=1pt];
\node[below] at (8.2,0){{\scriptsize $\max\big\{\f{n}{s},m\big\}$}};
\draw[dashed] (8.2,0)--(8.2,3.86);
\draw[thin] (8.2,3.86)--(8.2,4.8);

%=====================================p
\draw[fill] (7,2) circle[radius=1pt];
\node[left] at (7,2){{\scriptsize $1+s$}};
\draw[dashed] (7,2)--(11.8,2);
\draw[dashed] (7,0)--(10.05,2.54);
\draw[thin] (10.05,2.54)--(11.8,4);

\draw[domain=2.54:3.86,color=black,variable=\p] plot ({4.65+13.7/\p},\p);

\end{tikzpicture}
\caption{The global existence range in Theorem \ref{Global_Theorem} depending on some parameters.}
\label{fig.zone}
\end{center}
\end{figure}
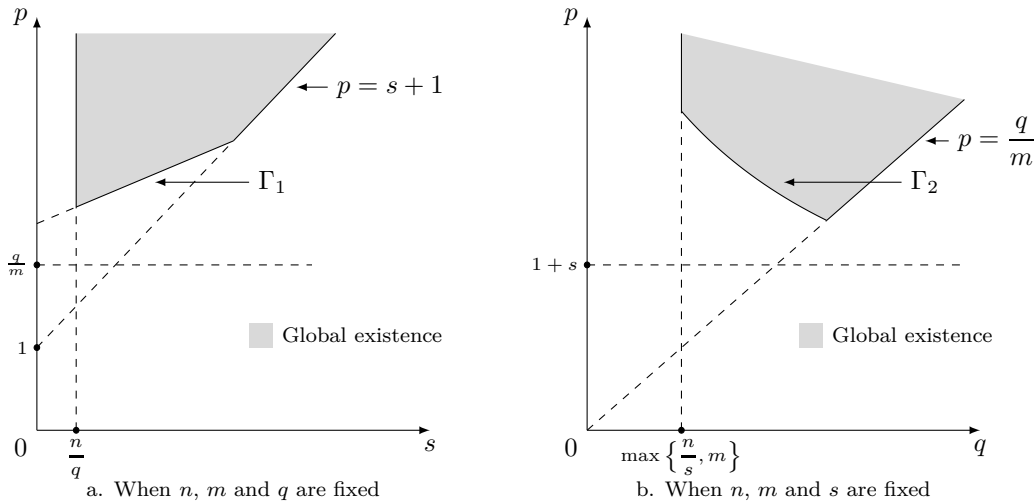
\noindent Here, $\Gamma_1$ and $\Gamma_2$ stand for the following graphs:
$$p= p(s)= 1+\f{n(1-\frac{m}{q})+ms}{n-2m(1+\kappa)} \ \  \text{and}\ \  p= p(q)= 1+\f{n(1-\frac{m}{q})+ms}{n-2m-(1-\frac{m}{q})[\frac{n}{2}-1]}$$
in Figure \ref{fig.zone}.a and Figure \ref{fig.zone}.b, respectively.
\end{remark}

\subsection{Preliminaries: Some useful estimates}
 For proving the global (in time) well-posedness result for the nonlinear problem \eqref{Equation_Main}, the following $L^{q}- L^{q}$ estimate with additional $L^{m}$ regularity for the Cauchy data from the corresponding linear equation \eqref{Equation_Main_Linear} is one of the most essential ingredients in our proof.
\begin{prop} \label{Linear_Estimates_Prop}
 Let $s\geqslant0$, $q\in (1,\ity)$ and $m\in [1,q)$. Then, the solutions to \eqref{Equation_Main_Linear} satisfy the following $(L^m \cap L^q)-L^q$ estimates:
\begin{align*}
    \|v(t,\cdot )\|_{\dot{H}^s_q}& \lesssim
        (1+t)^{-\frac{n}{2}(1-\frac{1}{r})+\frac{1}{2r}[\frac{n}{2}]-\frac{s}{2}} \|v_0\|_{L^m\cap H^s_q}+ (1+t)^{1-\frac{n}{2}(1-\frac{1}{r})+\frac{1}{2r}[\frac{n}{2}-1]-\frac{s}{2}} \|v_1\|_{L^m\cap H^{\max\{s-2,0\}}_q}
\end{align*}
 for all $n\geqslant 1$, where $1+ 1/q= 1/r+1/m$.
\end{prop}

\begin{proof}
 As a consequence of Propositions \ref{L^r_Low} and \ref{L^r_High}, we estimate the $L^q$ norm of solutions for low frequencies by the $L^m$ norm of the data, whereas their high frequency parts are controlled by using $L^q-L^q$ estimates with a suitable regularity of initial data. In this way, after the immediate application of Young's convolution inequality linked to Propositions \ref{L^r_Low} and \ref{L^r_High}, we may achieve the desired estimates.
\end{proof}

Furthermore, the behavior of the following integral will be also necessary in the sequel.

\begin{lemma} \label{UsefulIntegral}
Let $\alpha, \beta \in \R$. It holds
$$ \int_0^t (1+t-\tau)^{-\alpha}(1+\tau)^{-\beta}\rm{d}\tau \lesssim
\begin{cases}
(1+t)^{-\min\{\alpha, \beta\}} &\text{if}\ \ \max\{\alpha, \beta\}>1,\\
(1+t)^{-\min\{\alpha, \beta\}}\log(\mathrm{e}+t) &\text{if}\ \ \max\{\alpha, \beta\}=1,\\
(1+t)^{1-\alpha-\beta} &\text{if}\ \ \max\{\alpha, \beta\}<1.\\
\end{cases} $$
\end{lemma}

\begin{proof}
Thanks to the asymptotic behaviors of $1+t-\tau \approx 1+t$ for any $\tau \in [0,t/2]$ as well as $1+\tau \approx 1+t$ for any $\tau \in [t/2,t]$, we carry out the change of variables when needed to conclude
\begin{align*}
\int_0^t (1+t-\tau)^{-\alpha}(1+\tau)^{-\beta}\rm{d}\tau 
&\lesssim (1+t)^{-\alpha} \int_0^{t/2}(1+\tau)^{-\beta}\rm{d}\tau+ (1+t)^{-\beta} \int_0^{t/2}(1+\tau)^{-\alpha}\rm{d}\tau \\
&\lesssim (1+t)^{-\min\{\alpha, \beta\}} \int_0^{t/2}(1+\tau)^{-\max\{\alpha, \beta\}}\rm{d}\tau.
\end{align*}
Consequently, this completes the proof of Lemma \ref{UsefulIntegral}.
\end{proof}

\subsection{Philosophy of the proof}
Recalling the fundamental solutions $K_0(t,x)$ and $K_1(t,x)$ defined in Section \ref{Section_Lm-Lq}, we may express the solutions to the linear equation \eqref{Equation_Main_Linear} by
\begin{align*}
	v(t,x)=K_0(t,x) \ast_{(x)} v_0(x)+ K_1(t,x) \ast_{(x)} v_1(x).
\end{align*}
We may apply Duhamel's principle to derive the following integral equations associated with \eqref{Equation_Main}:
\begin{align*}
u(t,x)&= K_0(t,x) \ast_{(x)} u_0(x)+ K_1(t,x) \ast_{(x)} u_1(x) + \int_0^t K_1(t-\tau,x) \ast_{(x)} \Delta f\left(u(\tau,x);p\right) \rm{d}\tau\\
&=: u^{\rm ln}(t,x)+ u^{\rm nl}(t,x).
\end{align*}
Later, we will determine  the family $\{X(T)\}_{T>0}$ of evolution spaces appropriately. We introduce the operator $\Phi$ such that
$$\Phi:\,\, u(t,x)\in X(T)\to \Phi[u](t,x):= u^{\rm ln}(t,x)+ u^{\rm nl}(t,x) $$
for $T>0$. From the setting, to establish the global (in time) well-posedness for the nonlinear problem \eqref{Equation_Main}, we shall prove a fixed point of the operator $\Phi$ to be the solution. For this purpose, it suffices to verify the following two score inequalities:
\begin{align}
\|\Phi[u]\|_{X(T)}& \lesssim \|(u_0,u_1)\|_{\mathcal{D}}+ \|u\|^p_{X(T)}, \label{Phi_INE_1}\\
\|\Phi[u]-\Phi[\bar{u}]\|_{X(T)}& \lesssim \|u-\bar{u}\|_{X(T)} \left(\|u\|^{p-1}_{X(T)}+ \|\bar{u}\|^{p-1}_{X(T)}\right), \label{Phi_INE_2}
\end{align}
under some conditions for $p$. In the desired estimate \eqref{Phi_INE_2}, $u$ and $\bar{u}$ are two solutions to the nonlinear viscous Boussinesq equation. Subsequently, an employment of Banach's fixed point theorem leads to local (in time) existence results of large data solutions and global (in time) existence results of small data solutions simultaneously. This idea has been used in semilinear damped wave equations, for example, \cite{Palmieri-Reissig=2018,D'Abbicco-Palmieri=2021}. 

\subsection{Existence of unique global (in time) solution: Proof of Theorem \ref{Global_Theorem}}
	As mentioned in the previous part, to get started, we introduce the solution space by
	$$ X(T):= \ml{C}\left([0,T],H^s_q\right) $$
	for $T>0$, endowed with the corresponding norm
	$$ \|u\|_{X(T)}:= \sup_{0\leqslant \tau \leqslant T} \left((1+\tau)^{-1+\frac{n}{2}(1-\frac{1}{r})-\kappa}\|u(\tau,\cdot)\|_{L^q}+ (1+\tau)^{-1+\frac{n}{2}(1-\frac{1}{r})-\kappa+\frac{s}{2}}\|\,|D|^s u(\tau,\cdot)\|_{L^q}\right)$$
	with $1+1/q=1/r+1/m$.
	
	Let us now estimate the nonlinear term in some norms. For the sake of indicating the both inequalities \eqref{Phi_INE_1} and \eqref{Phi_INE_2}, we shall verify the following auxiliary ingredients for the nonlinear term:
	\begin{align}
	\|f\left(u(\tau,\cdot);p\right)\|_{L^m}\lesssim \|u(\tau,\cdot)\|^p_{L^{mp}} &\lesssim (1+\tau)^{-\frac{n}{2m}(p-1)+p(1+\kappa)}\|u\|^p_{X(T)}, \label{NL_Estimate_1}\\ 
	\|f\left(u(\tau,\cdot);p\right)\|_{L^q}\lesssim \|u(\tau,\cdot)\|^p_{L^{pq}} &\lesssim (1+\tau)^{-\frac{n}{2}(\frac{p}{m}-\frac{1}{q})+p(1+\kappa)}\|u\|^p_{X(T)}, \label{NL_Estimate_2}\\
	\|f\left(u(\tau,\cdot);p\right)\|_{\dot{H}^s_q} &\lesssim (1+\tau)^{-\frac{n}{2m}(p-1)+p(1+\kappa)-\frac{s}{2}}\|u\|^p_{X(T)}, \label{NL_Estimate_3}
	\end{align}
	for $\tau\in[0,T]$. Indeed, the application of the fractional Gagliardo-Nirenberg inequality from Proposition \ref{Fractional_GN} combined with the definition of the solution space gives immediately these estimates \eqref{NL_Estimate_1} and \eqref{NL_Estimate_2}, accompanied with
	$$ p \in \left[\frac{q}{m}, \infty \right) \ \ \text{and} \  \ s>\f{n}{q}, \text{ i.e. }  n< qs. $$
	In order to show \eqref{NL_Estimate_3}, we apply the fractional powers rule from Proposition \ref{Fractional_Powers} for $s\in (n/q,p)$ together with the fractional Sobolev embedding from Proposition \ref{Fractional_Embedding} for a suitable $0<s^* <n/q$ to obtain
    \begin{align*}
        \|f\left(u(\tau,\cdot);p\right)\|_{\dot{H}^s_q}\lesssim \big\||u(\tau,\cdot)|^p\big\|_{\dot{H}^s_q} &\lesssim \|u(\tau,\cdot)\|_{\dot{H}^s_q}\,\|u(\tau,\cdot)\|^{p-1}_{L^\ity} \\
        &\lesssim \|u(\tau,\cdot)\|_{\dot{H}^s_q}\big(\|u(\tau,\cdot)\|_{\dot{H}^{s^*}_q}+ \|u(\tau,\cdot)\|_{\dot{H}^s_q}\big)^{p-1}.
    \end{align*}
    After employing the fractional Gagliardo-Nirenberg inequality from Proposition \ref{Fractional_GN}, we derive
    \begin{align*}
         \|u(\tau,\cdot)\|_{\dot{H}^{s^*}_q} &\lesssim \|u(\tau,\cdot)\|^{1-\theta_1}_{L^q}\,\|u(\tau,\cdot)\|^{\theta_1}_{\dot{H}^s_q} \lesssim (1+\tau)^{1-\frac{n}{2}(1-\frac{1}{r})+\kappa-\frac{s^*}{2}}\|u\|_{X(T)}
    \end{align*}
    with $\theta_1= s^*/s$. Therefore, one has
    \begin{align*}
        \|f\left(u(\tau,\cdot);p\right)\|_{\dot{H}^s_q} &\lesssim (1+\tau)^{-\frac{np}{2}(1-\frac{1}{r})+p(1+\kappa)-\frac{s}{2}-(p-1)\frac{s^*}{2}} \|u\|^p_{X(T)} \\
        &\lesssim (1+\tau)^{-\frac{n}{2m}(p-1)+p(1+\kappa)-\frac{s}{2}}  \|u\|^p_{X(T)},
        \end{align*}
    where we took  $s^*= n/q-\e$ with a sufficiently small positive $\e$. Hence, our verifications for \eqref{NL_Estimate_1}-\eqref{NL_Estimate_3} are accomplished. Now, we are on the way to prove the inequalities \eqref{Phi_INE_1} and \eqref{Phi_INE_2}.\medskip
    
    First let us prove the inequality \eqref{Phi_INE_1}. Thanks to the statements in Proposition \ref{Linear_Estimates_Prop}, one may obviously arrive at
	    $$ \|u^{\rm ln} \|_{X(T)} \lesssim \|(u_0,u_1)\|_{\mathcal{D}} $$
	from the definition of the norm of $X(T)$. For this reason, we reduce the proof of \eqref{Phi_INE_1} to
	\begin{equation}
	    \|u^{\rm nl}\|_{X(T)} \lesssim \|u\|^p_{X(T)}. \label{Phi_INE_1*}
	 \end{equation}
    Next, our proof is separated into two steps as follows:
    
    \noindent \underline{Step 1}:  To estimate $\|u^{\rm nl}(t,\cdot)\|_{L^q}$, we use the $(L^m \cap L^q)- L^q$ estimates in Proposition \ref{Linear_Estimates_Prop} to get
	\begin{align*}
	    \|u^{\rm nl}(t,\cdot)\|_{L^q}&\lesssim \int_0^t \left\|\Delta K_1(t-\tau,x) \ast_x f\left(u(\tau,x);p\right)\right\|_{L^q}\rm{d}\tau \\
	    &\lesssim \int_0^t (1+t-\tau)^{-\frac{n}{2}(1-\frac{1}{r})+\kappa}\|f\left(u(\tau,\cdot);p\right)\|_{L^m\cap L^q} \rm{d}\tau \\
	    &\lesssim \|u\|^p_{X(T)} \int_0^t (1+t-\tau)^{-\frac{n}{2}(1-\frac{1}{r})+\kappa}(1+\tau)^{-\frac{n}{2m}(p-1)+p(1+\kappa)} \mathrm{d}\tau =: \|u\|^p_{X(T)}I_1(t),
	\end{align*}
	where we have applied the both estimates \eqref{NL_Estimate_1} and \eqref{NL_Estimate_2} in the last chain. Now, Lemma \ref{UsefulIntegral} comes into play in term of estimate for $I_1(t)$ by denoting
	    $$\alpha= \frac{n}{2}\left(1-\frac{1}{r}\right)-\kappa\ \  \text{and}\ \  \beta= \frac{n}{2m}(p-1)-p(1+\kappa). $$
	Here, one  recognizes that the condition \eqref{Exponent_Condition} follows
	    $$p> 1+\f{n(1-\frac{m}{q})+ms}{n-2m(1+\kappa)}> 1+\f{n}{n-2m(1+\kappa)}> 1+\f{2m(1+\kappa)}{n-2m(1+\kappa)}$$
	due to $s >n/q$ and $n> 2m(1+\kappa)$. This leads to $\beta>0$.
	\begin{itemize} 
    \item[$\bullet$] \textbf{Case 1}: Providing that 
    \begin{align*}
    	p>1+\frac{n\left(1-\frac{m}{q}\right)+m\,\max\{s,2\}}{n-2m(1+\kappa)},
    \end{align*}
    then it is clear to see $\max\{\alpha,\beta\}= \beta>0$, i.e. $\min\{\alpha,\beta\}= \alpha$. Thus, the application of Lemma \ref{UsefulIntegral} implies
    \begin{align*}
    	I_1(t) \lesssim (1+t)^{1-\alpha},
    \end{align*}

    which gives
    \begin{equation}\label{Estimate_1}
        \|u^{\rm nl}(t,\cdot)\|_{L^q}\lesssim (1+t)^{1-\frac{n}{2}(1-\frac{1}{r})+\kappa}\|u\|^p_{X(T)}. 
    \end{equation}
    \item[$\bullet$] \textbf{Case 2}: Providing that
    \begin{align*}
    1+\frac{n\left(1-\frac{m}{q}\right)+m\,\max\{s,2\}}{n-2m(1+\kappa)}\geqslant p>1+\frac{n\left(1-\frac{m}{q}\right)+ms}{n-2m(1+\kappa)},	
    \end{align*}
     then one realizes that $\max\{\alpha,\beta\}= \alpha$, i.e. $\min\{\alpha,\beta\}= \beta$. After employing Lemma \ref{UsefulIntegral}, we achieve
   \begin{align*}
   	I_1(t) &\lesssim \begin{cases}
   		(1+t)^{-\beta} &\text{if}\ \ \alpha>1\\
   		(1+t)^{-\beta}\log(\mathrm{e}+t) &\text{if}\ \ \alpha=1\\
   		(1+t)^{1-\alpha-\beta} &\text{if}\ \ \alpha<1\\
   	\end{cases}\\
   	&\lesssim (1+t)^{1-\alpha}
   \end{align*}
    by noticing that the relation $-\beta< 1-\alpha$ holds in the first line of the previous integral due to the fact 
    $$p> 1+\frac{n(1-\frac{m}{q})}{n-2m(1+\kappa)}. $$
    Again, it follows
    \begin{equation}\label{Estimate_2}
        \|u^{\rm nl}(t,\cdot)\|_{L^q}\lesssim (1+t)^{1-\frac{n}{2}(1-\frac{1}{r})+\kappa}\|u\|^p_{X(T)}. 
    \end{equation}
    \end{itemize}
    
	\noindent \underline{Step 2}:  To control $\|\,|D|^s u^{\rm nl}(t,\cdot)\|_{L^q}$, we will use the same strategy as in Step 1. In particular, applying the $(L^m \cap L^q)- L^q$ estimates in Proposition \ref{Linear_Estimates_Prop} one obtains
	\begin{align*}
	    \|\,|D|^s u^{\rm nl}(t,\cdot)\|_{L^q}&\lesssim \int_0^t \left\|\Delta K_1(t-\tau,x) \ast_x f\left(u(\tau,x);p\right)\right\|_{H^s_q}\rm{d}\tau \\
	    &\lesssim \int_0^t (1+t-\tau)^{-\frac{n}{2}(1-\frac{1}{r})+\kappa-\frac{s}{2}}\|f\left(u(\tau,\cdot);p\right)\|_{L^m\cap L^q\cap \dot{H}^s_q} \rm{d}\tau \\
	    &\lesssim \|u\|^p_{X(T)} \int_0^t (1+t-\tau)^{-\frac{n}{2}(1-\frac{1}{r})+\kappa-\frac{s}{2}}(1+\tau)^{-\frac{n}{2m}(p-1)+p(1+\kappa)} \mathrm{d}\tau =: \|u\|^p_{X(T)}I_2(t),
	\end{align*}
	where we have utilized the estimates from \eqref{NL_Estimate_1} to \eqref{NL_Estimate_3} at the last stage of the above chain inequality. Then, repeating some procedures as we did in Step $1$ we may claim
	\begin{align*}
		I_2(t)\lesssim(1+t)^{1-\alpha}
	\end{align*}
so that
	\begin{equation}\label{Estimate_3}
        \|\,|D|^s u^{\rm nl}(t,\cdot)\|_{L^q}\lesssim (1+t)^{1-\frac{n}{2}(1-\frac{1}{r})+\kappa-\frac{s}{2}}\|u\|^p_{X(t)}. 
    \end{equation}
    As a result, by the definition of the norm in $X(T)$ we combine the estimates from \eqref{Estimate_1} to \eqref{Estimate_3} to conclude that the inequality \eqref{Phi_INE_1*} holds.\medskip
    
	Our next aim is to turn the proof of the Lipschitz condition \eqref{Phi_INE_2}. Our proof is based on the similar approach to the  proof of the obtained inequality \eqref{Phi_INE_1*}. Indeed, we may estimate
	\begin{align*}
	     \left|f\left(u(\tau,\cdot);p\right)- f\left(\bar{u}(\tau,\cdot);p\right)\right|\approx \left||u(\tau,\cdot)|^p-|\bar{u}(\tau,\cdot)|^p\right|
	\end{align*}
	in the norms of $L^m$, $L^q$ and $\dot{H}^s_q$. Thanks to H\"{o}lder's inequality, one finds
	\begin{align*}
	    \|\,|u(\tau,\cdot)|^p-|\bar{u}(\tau,\cdot)|^p\|_{L^m} &\lesssim \|u(\tau,\cdot)- \bar{u}(\tau,\cdot)\|_{L^{mp}} \left(\|u(\tau,\cdot)\|^{p-1}_{L^{mp}}+\|\bar{u}(\tau,\cdot)\|^{p-1}_{L^{mp}}\right), \\ 
	    \|\,|u(\tau,\cdot)|^p- |\bar{u}(\tau,\cdot)|^p\|_{L^q} &\lesssim \|u(\tau,\cdot)- \bar{u}(\tau,\cdot)\|_{L^{pq}} \left(\|u(\tau,\cdot)\|^{p-1}_{L^{pq}}+\|\bar{u}(\tau,\cdot)\|^{p-1}_{L^{pq}}\right).
	\end{align*}
	Analogously to the proof of \eqref{NL_Estimate_1} and \eqref{NL_Estimate_2}, the applications of the fractional Gagliardo-Nirenberg inequality from Proposition \ref{Fractional_GN} to deal with the norms
	$$ \|u(\tau,\cdot)- \bar{u}(\tau,\cdot)\|_{L^\gamma}, \quad \|u(\tau,\cdot)\|_{L^\gamma}, \quad \|\bar{u}(\tau,\cdot)\|_{L^\gamma}, $$
	with $\gamma=mp$ and $\gamma=pq$ give the following estimates:
	\begin{align*}
	\|\,|u(\tau,\cdot)|^p-|\bar{u}(\tau,\cdot)|^p\|_{L^m} &\lesssim (1+\tau)^{-\frac{n}{2m}(p-1)+p(1+\kappa)}\|u-\bar{u}\|_{X(T)}\left(\|u\|^{p-1}_{X(T)}+\|\bar{u}\|^{p-1}_{X(T)}\right),\\
	\|\,|u(\tau,\cdot)|^p- |\bar{u}(\tau,\cdot)|^p\|_{L^q} &\lesssim (1+\tau)^{-\frac{n}{2}(\frac{p}{m}-\frac{1}{q})+p(1+\kappa)} \|u-\bar{u}\|_{X(T)}\left(\|u\|^{p-1}_{X(T)}+\|\bar{u}\|^{p-1}_{X(T)}\right).
	\end{align*}
	For another, let us devote to our attention in estimating the difference
	    $$\|\,|u(\tau,\cdot)|^p-|\bar{u}(\tau,\cdot)|^p\|_{\dot{H}^s_q}.$$
	According to the expression
	$$ |u(\tau,x)|^p-|\bar{u}(\tau,x)|^p=p\int_0^1 \left(u(\tau,x)-\bar{u}(\tau,x)\right)G\left(\omega u(\tau,x)+(1-\omega)\bar{u}(\tau,x)\right)\mathrm{d}\omega, $$
 where we set $G(u):=u|u|^{p-2}$, we deduce
	\begin{align*}
	\|\,|u(\tau,\cdot)|^p-|\bar{u}(\tau,\cdot)|^p\|_{\dot{H}^s_q}\lesssim \int_0^1 \left\|\big(u(\tau,\cdot)-\bar{u}(\tau,\cdot)\big)G\left(\omega u(\tau,\cdot)+(1-\omega)\bar{u}(\tau,\cdot)\right)\right\|_{\dot{H}^s_q}\mathrm{d}\omega .
	\end{align*}
	Using the fractional Leibniz rule from Proposition \ref{Fractional_Leibniz} we can proceed
	\begin{align}
	\|\,|u(\tau,\cdot)|^p-|\bar{u}(\tau,\cdot)|^p\|_{\dot{H}^s_q}
	& \lesssim \|u(\tau,\cdot)-\bar{u}(\tau,\cdot)\|_{\dot{H}^s_q} \int_0^1 \big\|G\big(\omega u(\tau,\cdot)+(1-\omega)\bar{u}(\tau,\cdot)\big)\big\|_{L^\ity}\mathrm{d}\omega \nonumber \\
	&\quad+ \big\|u(\tau,\cdot)-\bar{u}(\tau,\cdot)\|_{L^\ity} \int_0^1 \big\|G\big(\omega u(\tau,\cdot)+(1-\omega)\bar{u}(\tau,\cdot)\big)\big\|_{\dot{H}^s_q}\mathrm{d}\omega \nonumber \\
	& \lesssim \|u(\tau,\cdot)-\bar{u}(\tau,\cdot)\|_{\dot{H}^s_q} \left(\|u(\tau,\cdot)\|^{p-1}_{L^\ity}+ \|\bar{u}(\tau,\cdot)\|^{p-1}_{L^\ity}\right) \nonumber \\
	&\quad+ \|u(\tau,\cdot)-\bar{u}(\tau,\cdot)\|_{L^\ity} \int_0^1 \big\|G\big(u(\tau,\cdot)+(1-\omega)\bar{u}(\tau,\cdot)\big)\big\|_{\dot{H}^s_q}\mathrm{d}\omega. \label{Estimate_4}
	\end{align}
	After applying the fractional powers rule from Proposition \ref{Fractional_Powers} with $p>2$ and $s\in (n/q,p-1)$, one derives
	\begin{align*}
	\left\|G\left(u(\tau,\cdot)+(1-\omega)\bar{u}(\tau,\cdot)\right)\right\|_{\dot{H}^s_q}
	& \lesssim \|\omega u(\tau,\cdot)+(1-\omega) \bar{u}(\tau,\cdot)\|_{\dot{H}^s_q}\, \|\omega u(\tau,\cdot)+(1-\omega) \bar{u}(\tau,\cdot)\|^{p-2}_{L^\ity}.
	\end{align*}	
	The employment of the fractional Sobolev embedding from Proposition \ref{Fractional_Embedding} carrying a suitable parameter $s^* <n/q$ yields the following chain of estimates in the $L^\ity$ nomrs:
	\begin{align*}
        \|u(\tau,\cdot)\|_{L^\ity} &\lesssim \|u(\tau,\cdot)\|_{\dot{H}^{s^*}_q}+ \|u(\tau,\cdot)\|_{\dot{H}^s_q},\\ 
        \|\bar{u}(\tau,\cdot)\|_{L^\ity} &\lesssim \|\bar{u}(\tau,\cdot)\|_{\dot{H}^{s^*}_q}+ \|\bar{u}(\tau,\cdot)\|_{\dot{H}^s_q}, \\
        \|u(\tau,\cdot)-\bar{u}(\tau,\cdot)\|_{L^\ity} &\lesssim \|u(\tau,\cdot)-\bar{u}(\tau,\cdot)\|_{\dot{H}^{s^*}_q}+ \|u(\tau,\cdot)-\bar{u}(\tau,\cdot)\|_{\dot{H}^s_q}, \\
        \|\omega u(\tau,\cdot)+(1-\omega) \bar{u}(\tau,\cdot)\|_{L^\ity} &\lesssim \|\omega u(\tau,\cdot)+(1-\omega) \bar{u}(\tau,\cdot)\|_{\dot{H}^{s^*}_q}+ \|\omega u(\tau,\cdot)+(1-\omega) \bar{u}(\tau,\cdot)\|_{\dot{H}^s_q}.
    \end{align*}
    With the help of the fractional Gagliardo-Nirenberg inequality from Proposition \ref{Fractional_GN}, we have
	\begin{align*}
	\|u(\tau,\cdot)\|_{\dot{H}^{s^*}_q}&\lesssim (1+\tau)^{1-\frac{n}{2}(1-\frac{1}{r})+\kappa-\frac{s^*}{2}}\|u\|_{X(T)}, \\
	\|\bar{u}(\tau,\cdot)\|_{\dot{H}^{s^*}_q}&\lesssim (1+\tau)^{1-\frac{n}{2}(1-\frac{1}{r})+\kappa-\frac{s^*}{2}}\|\bar{u}\|_{X(T)}, \\
	\|u(\tau,\cdot)-\bar{u}(\tau,\cdot)\|_{\dot{H}^{s^*}_q}&\lesssim (1+\tau)^{1-\frac{n}{2}(1-\frac{1}{r})+\kappa-\frac{s^*}{2}}\|u(\tau,\cdot)-\bar{u}(\tau,\cdot)\|_{X(T)}, \\
	\|\omega u(\tau,\cdot)+(1-\omega) \bar{u}(\tau,\cdot)\|_{\dot{H}^{s^*}_q}&\lesssim (1+\tau)^{1-\frac{n}{2}(1-\frac{1}{r})+\kappa-\frac{s^*}{2}}\|\omega u(\tau,\cdot)+(1-\omega) \bar{u}(\tau,\cdot)\|_{X(T)}.
	\end{align*}
	We next plug the above obtained estimates into \eqref{Estimate_4} to state
	\begin{align*}
	\|\,|u(\tau,\cdot)|^p-|\bar{u}(\tau,\cdot)|^p\|_{\dot{H}^s_q} &\lesssim (1+\tau)^{-\frac{np}{2}(1-\frac{1}{r})+p(1+\kappa)-\frac{s}{2}-(p-1)\frac{s^*}{2}} \|u-\bar{u}\|_{X(T)} \left(\|u\|^{p-1}_{X(T)}+ \|\bar{u}\|^{p-1}_{X(T)}\right) \\
	&\lesssim (1+\tau)^{-\frac{n}{2m}(p-1)+p(1+\kappa)-\frac{s}{2}} \|u-\bar{u}\|_{X(T)} \left(\|u\|^{p-1}_{X(T)}+ \|\bar{u}\|^{p-1}_{X(T)}\right),
	\end{align*}
	after taking $s^*= n/q-\e$ again with a sufficiently small positive $\e$. Finally, following an analogous manner to the proof of \eqref{Phi_INE_1*} we may affirm that the inequality \eqref{Phi_INE_2} is valid.  Summarizing, Theorem \ref{Global_Theorem} is proved completely.

\appendix
\section{Tools in the Fourier analysis}\label{Section_Appendix_A}
In this section, we will recall some tools from the Fourier analysis that have been applied in the treatments of the linear Cauchy problem.
\begin{prop}[Modified Bessel functions, Proposition 24.2.1 in \cite{Ebert-Reissig=book}] \label{Modified-Bessel-Funct}
	Let $h \in L^\theta$ with $\theta\in [1,2]$, be a radial function. Then, the Fourier transform $\ml{F}(h)$ is also a radial function and it satisfies
	\begin{align*}
		\ml{F}(h)(\xi)= c \int_0^\ity g(r) r^{n-1} \tilde{\mathcal{J}}_{\frac{n}{2}-1}(r|\xi|)\mathrm{d}r\ \ \mbox{carrying}\ \  g(|x|):= h(x),
	\end{align*}
	where $\tilde{\mathcal{J}}_\mu(s):=s^{-\mu}\mathcal{J}_\mu(s)$ is the so-called modified Bessel function with the Bessel function $\ml{J}_\mu(s)$ and a non-negative integer $\mu$. Moreover, the following properties hold for the modified Bessel functions:
	\begin{enumerate}[(1)]
		\item $s\,{\rm d}_s\tilde{\mathcal{J}}_\mu(s)= \tilde{\mathcal{J}}_{\mu-1}(s)-2\mu \tilde{\mathcal{J}}_\mu(s)$,
		\item ${\rm d}_s\tilde{\mathcal{J}}_\mu(s)= -s\tilde{\mathcal{J}}_{\mu+1}(s)$,
		\item $\tilde{\mathcal{J}}_{-\frac{1}{2}}(s)= \sqrt{\f{2}{\pi}}\cos s$ \ and \ $\tilde{\mathcal{J}}_{\frac{1}{2}}(s)= \sqrt{\f{2}{\pi}} \f{\sin s}{s}$,
		\item $|\tilde{\mathcal{J}}_\mu(s)| \leqslant C\mathrm{e}^{\pi|\fontshape{n}\selectfont\text{Im}\mu|}\ \,\,$  if \ $s \leqslant 1,$ \ and\ $\ \,\,\tilde{\mathcal{J}}_\mu(s)= Cs^{-\frac{1}{2}}\displaystyle{\cos \left(s-\frac{\mu}{2}\pi- \frac{\pi}{4} \right)} +\mathcal{O}(|s|^{-\frac{3}{2}})$ \ if \ $|s|\geqslant 1$,
		\item $\tilde{\mathcal{J}}_{\mu+1}(r|x|)= -\f{1}{r|x|^2}\partial_r \tilde{\mathcal{J}}_\mu(r|x|)$\ with\ $r \ne 0$, $x \ne 0$.
	\end{enumerate}
\end{prop}

\begin{prop}[Fa\`{a} di Bruno's formula, \cite{FrancescoBruno1857}] \label{FadiBruno'sformula1}
	Let $h(\psi(x)):= (h\circ \psi)(x)$ with $x\in \R$. Then, a generalization of the chain rule to higher-order derivatives holds
	\begin{align*}
		\frac{{\rm d}^n}{{\rm d}x^n}h(\psi(x))= \sum _{\substack{1\cdot m_1+\cdots+ n\cdot m_n =n\\m_j\geqslant 0}}\frac{n!}{m_1! 1!^{m_1} \cdots m_n! n!^{m_n}}h^{(m_1+\cdots+m_n)}(\psi(x)) \prod_{j=1}^n \left(\psi^{(j)}(x) \right)^{m_j}.
	\end{align*}
\end{prop}
\section{Tools in the harmonic analysis}
In this section, we will list several new tools from the harmonic analysis, which have been used in Section \ref{Section_GESDS} to deal with the nonlinear term in different norms.

\begin{prop}[Fractional Gagliardo-Nirenberg inequality, \cite{Hajaiej-Ozawa-Wang2011}] \label{Fractional_GN}
Let $1<r,\,r_1,\,r_2<\infty$, $\sigma >0$ and $a\in [0,\sigma)$. Then, it holds
$$ \|u\|_{\dot{H}^{a}_r} \lesssim \|u\|_{L^{r_1}}^{1-\theta}\, \|u\|_{\dot{H}^{\sigma}_{r_2}}^\theta, $$
where $$\theta=\theta(r,r_1,r_2,a,\sigma,n)=\left(\frac{1}{r_1}-\frac{1}{r}+\frac{a}{n}\right)\big/\left(\frac{1}{r_1}-\frac{1}{r_2}+\frac{\sigma}{n}\right)$$ and $a/\sigma\leqslant \theta\leqslant 1$.
\end{prop}

\begin{prop}[Fractional powers rule, \cite{Run-Sickel1996}] \label{Fractional_Powers}
Let $p>1$, $1< q <\infty$ and $a \in (n/q,p)$. Let us denote by $\ml{G}(u)$, one of the functions $|u|^p,\, \pm |u|^{p-1}u$. Then, the following estimates hold:
$$\|\ml{G}(u)\|_{H^{a}_q}\lesssim \|u\|_{H^{a}_q}\, \|u\|_{L^\infty}^{p-1}\quad\text{ and }\quad \| \ml{G}(u)\|_{\dot{H}^{a}_q}\lesssim \| u\|_{\dot{H}^{a}_q}\, \|u\|_{L^\infty}^{p-1}. $$
\end{prop}

\begin{prop}[A fractional Sobolev embedding, Corollary 5.2 in \cite{Dao-Reissig=2019-2}] \label{Fractional_Embedding}
Let $1< q< \ity$ and $0< s_1< n/q< s_2$. Then, we have
$$\|u\|_{L^\ity} \lesssim \|u\|_{\dot{H}^{s_1}_q}+ \|u\|_{\dot{H}^{s_2}_q}. $$
\end{prop}

\begin{prop}[Fractional Leibniz rule, \cite{Grafakos2004}] \label{Fractional_Leibniz}
Let us assume $s>0$, $1\leqslant r \leqslant \infty$ and $1<r_1,\,r_2,\,r_3,\,r_4 \leqslant \infty$ satisfying the relation
\[ \frac{1}{r}=\frac{1}{r_1}+\frac{1}{r_2}=\frac{1}{r_3}+\frac{1}{r_4}.\]
Then, the following inequality holds:
$$ \|u v\|_{H^s_r}\lesssim \|u\|_{H^s_{r_1}}\,\|v\|_{L^{r_2}}+\|u\|_{L^{r_3}}\,\|v\|_{H^s_{r_4}}. $$
\end{prop}

\section*{Acknowledgements}
The first author was supported by the China Postdoctoral Science Foundation (Grant No. 2021T140450 and No. 2021M692084). The authors thank Ryo Ikehata (Hiroshima University) for the suggestions in studying the viscous Boussinesq equation.

% ------------------------------------------------------------------------

\begin{thebibliography}{99}
%%%%%%
\bibitem{Baro-Figu-Himo=2019}
\newblock R.F. Barostichi, R.O. Figueira, A.A. Himonas,
\newblock Well-posedness of the good Boussinesq equation in analytic Gevrey spaces and time regularity.
\newblock \emph{J. Differential Equations} \textbf{267} (2019), no. 5, 3181--3198.
%%%%%%
\bibitem{Biler=1989}
\newblock P. Biler, 
\newblock Regular decay of solutions of strongly damped nonlinear hyperbolic equations. 
\newblock \emph{Appl. Anal.} \textbf{32} (1989), no. 3-4, 277--285. 
%%%%%%
\bibitem{Biler=1991}
\newblock P. Biler,
\newblock Time decay of solutions of semilinear strongly damped generalized wave equations.
\newblock \emph{Math. Methods Appl. Sci.} \textbf{14} (1991), no. 6, 427--443.
%%%%%%
\bibitem{Bona-Sachs=1988}
\newblock J.L. Bona, R.L. Sachs,
\newblock Global existence of smooth solutions and stability of solitary waves for a generalized Boussinesq equation.
\newblock \emph{Comm. Math. Phys.} \textbf{118} (1988), no. 1, 15--29.
%%%%%%
\bibitem{Boussinesq-01}
\newblock J.V. Boussinesq,
\newblock Th\'eorie de l'intumescence liquide appel\'ee onde solitaire ou de translation, se propageant dans un canal rectangulaire.
\newblock \emph{Comp. Rend. Hebd. Seances l'Acad. Sci.} \textbf{72} (1871), 755--759.
%%%%%%
\bibitem{Boussinesq-02}
\newblock J.V. Boussinesq,
\newblock Th\'eorie g\'en\'erale des mouvements qui sont propag\'es dans un canal rectangulaire horizontal.
\newblock \emph{Comp. Rend. Hebd. Seances l'Acad. Sci.} \textbf{73} (1871), 256--260.
%%%%%%
\bibitem{Boussinesq-03}
\newblock J.V. Boussinesq,
\newblock Th\'eorie des ondes et des remous qui se propagent le long d'un canal rectangulaire horizontal, en communiquant au liquide
contenu dans ce canal des vitesses sensiblement pareilles de la surface au fond.
\newblock \emph{J. Math. Pures Appl.} \textbf{17} (1872), 55--108.
%%%%%%
\bibitem{Cho-Ozawa=2007}
\newblock Y. Cho, T. Ozawa,
\newblock On small amplitude solutions to the generalized Boussinesq equations.
\newblock \emph{Discrete Contin. Dyn. Syst.} \textbf{17} (2007), no. 4, 691--711.
%%%%%%
\bibitem{D'Abbicco-Palmieri=2021}
\newblock M. D'Abbicco, A. Palmieri,
\newblock A note on $L^p$--$L^q$ estimates for semilinear critical dissipative Klein-Gordon equations.
\newblock \emph{J. Dynam. Differential Equations} \textbf{33} (2021), no. 1, 63--74.
%%%%%%
\bibitem{Dao-Reissig=2019-1}
\newblock T.A. Dao, M. Reissig, 
\newblock An application of $L^1$ estimates for oscillating integrals to parabolic like semi-linear structurally damped $\sigma$-evolution models. \newblock \emph{J. Math. Anal. Appl.} \textbf{476} (2019), 426--463.
%%%%%%
\bibitem{Dao-Reissig=2019-2}
\newblock T.A. Dao, M. Reissig, 
\newblock $L^1$ estimates for oscillating integrals and their applications to semi-linear models with $\sigma$-evolution like structural damping. \newblock \emph{Discrete Contin. Dyn. Syst.} \textbf{39} (2019), no. 9, 5431--5463.
%%%%%%
\bibitem{Deift-Tomei-Trub=1982}
\newblock P. Deift, C. Tomei, E. Trubowitz,
\newblock Inverse scattering and the Boussinesq equation.
\newblock \emph{Comm. Pure Appl. Math.} \textbf{35} (1982), no. 5, 567--628.
%%%%%%
\bibitem{Ebert-Reissig=book} 
\newblock M.R. Ebert, M. Reissig, 
\newblock \emph{Methods for partial differential equations. Qualitative properties of solutions, phase space analysis, semilinear models}.
\newblock Birkh\"auser/Springer, Cham, 2018.
%%%%%%
\bibitem{FrancescoBruno1857} 
\newblock C.F. Fa\`a di Bruno,
\newblock Note sur un nouvelle formule de calcul diff\'erentiel.
\newblock \emph{Quart. J. Math.} \textbf{1} (1857), 359--360
%%%%%%
\bibitem{Farah=2009}
\newblock L.G. Farah,
\newblock Local solutions in Sobolev spaces with negative indices for the good Boussinesq equation.
\newblock \emph{Comm. Partial Differential Equations} \textbf{34} (2009), no. 1-3, 52--73.
%%%%%%
\bibitem{Farah=2009-02}
\newblock L.G. Farah, 
\newblock Local solutions in Sobolev spaces and unconditional well-posedness for the generalized Boussinesq equation.
\newblock \emph{Commun. Pure Appl. Anal.} \textbf{8} (2009), no. 5, 1521--1539. 
%%%%%%
\bibitem{Farah-Linares=2010}
\newblock L.G. Farah, F. Linares, 
\newblock Global rough solutions to the cubic nonlinear Boussinesq equation.
\newblock \emph{J. Lond. Math. Soc. (2)} \textbf{81} (2010), no. 1, 241--254.
%%%%%%
\bibitem{Grafakos2004} 
\newblock L. Grafakos, 
\newblock \emph{Classical and modern Fourier analysis}.
\newblock Prentice Hall, 2004.
%%%%%%
\bibitem{Hajaiej-Ozawa-Wang2011}
\newblock H. Hajaiej, L. Molinet, T. Ozawa, B. Wang,
\newblock Necessary and sufficient conditions for the fractional Gagliardo-Nirenberg inequalities and applications to Navier-Stokes and generalized boson equations.
\newblock \emph{Harmonic analysis and nonlinear partial differential equations}, RIMS Kokyuroku Bessatsu, B26, Res.Inst.Math.Sci. (RIMS), Kyoto, (2011), 159--175.
%%%%%%
\bibitem{Holmes-1995}
\newblock M.H. Holmes,
\newblock \emph{Introduction to perturbation methods}. 
\newblock Texts in Applied Mathematics, 20. Springer-Verlag, New York, 1995.
%%%%%%
\bibitem{Levine=1974}
\newblock H.A. Levine,
\newblock Instability and nonexistence of global solutions to nonlinear wave equations of the form $P u_{tt}=-Au+F(u)$.
\newblock \emph{Trans. Amer. Math. Soc.} \textbf{192} (1974), 1--21.
%%%%%%
\bibitem{Linares=1993}
\newblock F. Linares,
\newblock Global existence of small solutions for a generalized Boussinesq equation.
\newblock \emph{J. Differential Equations} \textbf{106} (1993), no. 2, 257--293.
%%%%%%
\bibitem{Linares-Scialom=1995} 
\newblock F. Linares, M. Scialom, 
\newblock Asymptotic behavior of solutions of a generalized Boussinesq type equation.
\newblock \emph{Nonlinear Anal.} \textbf{25} (1995), no. 11, 1147--1158.
%%%%%%
\bibitem{Liu=1995}
\newblock Y. Liu,
\newblock Instability and blow-up of solutions to a generalized Boussinesq equation.
\newblock \emph{SIAM J. Math. Anal.} \textbf{26} (1995), no. 6, 1527--1546.
%%%%%%
\bibitem{Liu=1997}
\newblock Y. Liu, 
\newblock Decay and scattering of small solutions of a generalized Boussinesq equation.
\newblock \emph{J. Funct. Anal.} \textbf{147} (1997), no. 1, 51--68.
%%%%%%
\bibitem{Liu-Xu=2008}
\newblock Y. Liu, R. Xu,
\newblock Global existence and blow up of solutions for Cauchy problem of generalized Boussinesq equation.
\newblock \emph{Phys. D} \textbf{237} (2008), no. 6, 721--731.
%%%%%%
\bibitem{Liu-Wang=2019}
\newblock G. Liu, W. Wang, 
\newblock Inviscid limit for the damped Boussinesq equation.
\newblock \emph{J. Differential Equations} \textbf{267} (2019), no. 9, 5521--5542. 
%%%%%%
\bibitem{Liu-Wang=2020}
\newblock G. Liu, W. Wang,
\newblock Decay estimates for a dissipative-dispersive linear semigroup and application to the viscous Boussinesq equation.
\newblock \emph{J. Funct. Anal.} \textbf{278} (2020), no. 7, 108413, 21 pp.
%%%%%%
\bibitem{Liu-Wang=2014}
\newblock M. Liu, W. Wang,
\newblock Global existence and pointwise estimates of solutions for the multidimensional generalized Boussinesq-type equation.
\newblock \emph{Commun. Pure Appl. Anal.} \textbf{13} (2014), no. 3, 1203--1222.
%%%%%%
\bibitem{Lin-Wu-Loxton=2009}
\newblock Q. Lin, Y.H. Wu, R. Loxton,
\newblock On the Cauchy problem for a generalized Boussinesq equation.
\newblock \emph{J. Math. Anal. Appl.} \textbf{353} (2009), no. 1, 186--195.
%%%%%%
\bibitem{Ikehata-Soga}
\newblock R. Ikehata, M. Soga,
\newblock Asymptotic profiles for a strongly damped plate equation with lower order perturbation.
\newblock \emph{Commun. Pure Appl. Anal.} \textbf{14} (2015), no. 5, 1759--1780.
%%%%%%
\bibitem{Manoranjan-Ortega-Sanz=1988}
\newblock V.S. Manoranjan, T. Ortega, J.M. Sanz-Serna,
\newblock Soliton and antisoliton interactions in the good Boussinesq equation.
\newblock \emph{J. Math. Phys.} \textbf{29} (1988), no. 9, 1964--1968.
%%%%%%
\bibitem{Masmoudi=2007}
\newblock N. Masmoudi,
\newblock Remarks about the inviscid limit of the Navier-Stokes system.
\newblock \emph{Comm. Math. Phys.} \textbf{270} (2007), no. 3, 777--788.
%%%%%%
\bibitem{Nara-Reissig=2013}
\newblock T. Narazaki, M. Reissig, 
\newblock $L^1$ estimates for oscillating integrals related to structural damped wave models. \newblock \emph{Studies in phase space analysis with applications to PDEs}, 215--258, Progr. Nonlinear Differential Equations Appl., 84, Birkh\"auser/Springer, New York, 2013.
%%%%%%
\bibitem{Palmieri-Reissig=2018}
\newblock A. Palmieri, M. Reissig, 
\newblock Semi-linear wave models with power non-linearity and scale-invariant time-dependent mass and dissipation, II.
\newblock \emph{Math. Nachr.} \textbf{291} (2018), no. 11-12, 1859--1892.
%%%%%%
\bibitem{Russell}
\newblock J.S. Russell,
\newblock Report on the committee on waves, in: Report of the Seventh Meeting (1837) of British Association for the Advancement of
Science.
\newblock \emph{Liverpool, John Murray, London}, 1838, 417--496.
%%%%%%
\bibitem{Run-Sickel1996}
\newblock T. Runst,  W. Sickel,
\newblock \emph{Sobolev spaces of fractional order, Nemytskij operators, and nonlinear partial differential equations, De Gruyter series in nonlinear analysis and applications.}
\newblock Walter de Gruyter $\&$ Co., Berlin, 1996.
%%%%%%
\bibitem{Sachs=1990}
\newblock R.L. Sachs,
\newblock On the blow-up of certain solutions of the good Boussinesq equation.
\newblock \emph{Appl. Anal.} \textbf{36} (1990), no. 3-4, 145--152.
%%%%%%
\bibitem{Shibata=2000}
\newblock Y. Shibata,
\newblock On the rate of decay of solutions to linear viscoelastic equation. 
\newblock \emph{Math. Methods Appl. Sci.} \textbf{23} (2000), no. 3, 203--226. 
%%%%%%
\bibitem{Straughan=1992}
\newblock B. Straughan, 
\newblock Global nonexistence of solutions to some Boussinesq type equations.
\newblock \emph{J. Math. Phys. Sci.} \textbf{26} (1992), no. 2, 155--164.
%%%%%%
\bibitem{Tsutsumi-Matahashi=1991}
\newblock M. Tsutsumi, T. Matahashi, 
\newblock On the Cauchy problem for the Boussinesq type equation.
\newblock \emph{Math. Japon.} \textbf{36} (1991), no. 2, 371--379.
%%%%%%
\bibitem{Varlamov=1996-02} 
\newblock V. Varlamov, 
\newblock On the Cauchy problem for the damped Boussinesq equation. 
\newblock \emph{Differential Integral Equations} \textbf{9} (1996), no. 3, 619--634.
%%%%%%
\bibitem{Varlamov=1996}
\newblock V. Varlamov, 
\newblock Existence and uniqueness of a solution to the Cauchy problem for the damped Boussinesq equation. 
\newblock \emph{Math. Methods Appl. Sci.} \textbf{19} (1996), no. 8, 639--649.
%%%%%%
\bibitem{Wang=2013}
\newblock Y. Wang,
\newblock Asymptotic decay estimate of solutions to the generalized damped Bq equation.
\newblock \emph{J. Inequal. Appl.} 2013, 2013:323, 12 pp.
%%%%%%
\bibitem{Xu-Luo-Shen-Huang=2017}
\newblock R. Xu, Y. Luo, J. Shen, S. Huang,
\newblock Global existence and blow up for damped generalized Boussinesq equation.
\newblock \emph{Acta Math. Appl. Sin. Engl. Ser.} \textbf{33} (2017), no. 1, 251--262.
%%%%%%
\bibitem{Xue=2006}
\newblock R. Xue,
\newblock Local and global existence of solutions for the Cauchy problem of a generalized Boussinesq equation.
\newblock \emph{J. Math. Anal. Appl.} \textbf{316} (2006), no. 1, 307--327.
%%%%%%
\bibitem{Yang-Guo=2008}
\newblock Z. Yang, B. Guo,
\newblock Cauchy problem for the multi-dimensional Boussinesq type equation.
\newblock \emph{J. Math. Anal. Appl.} \textbf{340} (2008), no. 1, 64--80.
%%%%%%
\end{thebibliography}
\end{document}